\documentclass{amsart}

\usepackage{amssymb,amsmath,leqno}
\usepackage[colorlinks,plainpages,backref]{hyperref}
\usepackage{mathrsfs}
\usepackage[neveradjust]{paralist}
\usepackage[all]{xypic}
\xyoption{dvips}
\usepackage{epic,eepic}
\usepackage{url,epsfig,psfrag}

\theoremstyle{definition}
\newtheorem{ntn}{Notation}[section]
\newtheorem{dfn}[ntn]{Definition}
\theoremstyle{plain}
\newtheorem{lem}[ntn]{Lemma}
\newtheorem{prp}[ntn]{Proposition}
\newtheorem{thm}[ntn]{Theorem}
\newtheorem{cor}[ntn]{Corollary}
\newtheorem{cnj}[ntn]{Conjecture}

\theoremstyle{remark}
\newtheorem*{clm}{Claim}
\newtheorem{que}[ntn]{Question}
\newtheorem{rmk}[ntn]{Remark}
\newtheorem{exa}[ntn]{Example}

\newenvironment{stp}[2]{\smallskip {\bf Step #1}: \emph{#2}\smallskip}

\newcommand{\del}{\partial}
\newcommand{\dff}{\frac{\de f}{f}}
\newcommand{\de}{{\mathrm d}}
\newcommand{\dR}{{\mathrm dR}}
\newcommand{\eps}{\varepsilon}

\newcommand{\into}{\hookrightarrow}
\newcommand{\onto}{\twoheadrightarrow}

\renewcommand{\to}{\longrightarrow}
\newcommand{\nil}{{\mathrm{nil}}}
\newcommand{\red}{{\mathrm{red}}}
\newcommand{\reg}{{\mathrm{reg}}}

\newcommand{\minus}{\smallsetminus}
\newcommand{\ydx}{y\, {\mathrm d}x}

\newcommand{\calA}{\mathscr{A}}
\newcommand{\calC}{\mathscr{C}}

\newcommand{\calD}{\mathscr{D}}

\newcommand{\calI}{\mathscr{I}}

\newcommand{\calK}{\mathscr{K}}
\newcommand{\calL}{\mathscr{L}}
\newcommand{\calM}{\mathscr{M}}

\newcommand{\calO}{\mathscr{O}}
\newcommand{\calP}{\mathscr{P}}
\newcommand{\calR}{\mathscr{R}}

\newcommand{\frakm}{{\mathfrak{m}}}

\newcommand{\frakq}{{\mathfrak{q}}}
\newcommand{\frakn}{{\mathfrak{n}}}

\newcommand{\frakS}{{\mathfrak{S}}}

\newcommand{\CC}{\mathbb{C}}

\newcommand{\NN}{\mathbb{N}}
\newcommand{\PP}{\mathbb{P}}
\newcommand{\QQ}{\mathbb{Q}}

\newcommand{\ZZ}{\mathbb{Z}}

\newcommand{\x}{{\mathfrak x}}

\renewcommand{\bar}{\overline}

\DeclareMathOperator{\ann}{ann}
\DeclareMathOperator{\codim}{codim}
\DeclareMathOperator{\depth}{depth}

\DeclareMathOperator{\Der}{Der}
\DeclareMathOperator{\Div}{Div}
\DeclareMathOperator{\charVar}{charV}

\DeclareMathOperator{\gr}{gr}

\DeclareMathOperator{\hight}{ht}
\DeclareMathOperator{\Hom}{Hom}

\DeclareMathOperator{\image}{im}
\DeclareMathOperator{\Jac}{Jac}
\DeclareMathOperator{\lcm}{lcm}

\DeclareMathOperator{\modulo}{mod\,}
\DeclareMathOperator{\pdim}{pdim}
\DeclareMathOperator{\Proj}{Proj}

\DeclareMathOperator{\rk}{rk}

\DeclareMathOperator{\sgn}{sgn}

\DeclareMathOperator{\Spec}{Spec}
\DeclareMathOperator{\Sym}{Sym}

\DeclareMathOperator{\Var}{Var}

\def\C{{C}}

\def\schluss{\hfill\ensuremath{\diamond}}

\def\mustata{Musta\c t\u a}

\def\comment#1{}

\def\logarithmic{{Liouville\ }}
\def\Logarithmic{{Liouville\ }}

\begin{document}

\title[
The $D$-module of $f^s$
]
{
The Jacobian module, the Milnor fiber, and the $D$-module generated
by $f^s$
}


\author{Uli Walther}
\address{
U. Walther\\
Purdue University\\
Dept.\ of Mathematics\\
150 N.\ University St.\\
West Lafayette, IN 47907\\
USA}
\email{walther@math.purdue.edu}


\begin{abstract}
For a germ $f$ on a complex manifold $X$, we introduce a complex
derived from the Liouville form acting on logarithmic differential
forms, and give an exactness criterion. We use this \logarithmic
complex to connect properties of the $D$-module generated by $f^s$ to
homological data of the Jacobian ideal; specifically we show that for
a large class of germs the annihilator of $f^s$ is generated by
derivations.  Through local cohomology, we connect the cohomology of
the Milnor fiber to the Jacobian module via logarithmic differentials.

In particular, we consider (not necessarily reduced) hyperplane
arrangements: we prove a conjecture of Terao on the annihilator of
$1/f$; we confirm in many cases a corresponding conjecture on the
annihilator of $f^s$ but we disprove it in general; we show that the
Bernstein--Sato polynomial of an arrangement is not determined by its
intersection lattice; we prove that arrangements for which the
annihilator of $f^s$ is generated by derivations
 fulfill the Strong Monodromy Conjecture, and that this
includes as very special cases all arrangements of Coxeter and of
crystallographic type, and all multi-arrangements in dimension 3.
\end{abstract}



\thanks{Support by the National Science Foundation under grant
  1401392-DMS is gratefully acknowledged. The author would also like
  to thank MSRI for its hospitality when he was in residence during
  the Special Year in Commutative and Non-commutative Algebra in 2012/13.}


\maketitle
\setcounter{tocdepth}{1}
\tableofcontents
\numberwithin{equation}{section}

\section{Introduction}

Let $X$ be a smooth analytic space or $\CC$-scheme with structure
sheaf $\calO_X$ and sheaf of $\CC$-linear differential operators
$\calD_X$ on $\calO_X$.  
Throughout,
$f\in\calO_X$ will be a regular analytic non-constant function on
$X$, not necessarily reduced, with divisor $\Div(f)$. 
To any such $f$ one can attach
several invariants that measure the singularity structure of the
hypersurface $\Var(f)$. In this article, we are particularly interested in
the following:
\begin{enumerate}
\item The (parametric) annihilator $\ann_{\calD_X[s]}(f^s)$: if $s$
  is a new variable, one associates to $f$ the multi-valued function
  $f^s$ on the locus $\{\x\in X\mid f(\x)\not =0\}$. The free
  $\calO_X[f^{-1},s]$-module generated by the symbol $f^s$ allows a left
  $\calD_X[s]$-structure via the inherited $\calO_X[s]$-structure and the rule
\begin{eqnarray}\label{eqn-intro-fs}
\delta\bullet
(\frac{g}{f^k}f^s)&=&\delta\bullet(\frac{g}{f^k})f^s+
\frac{sg}{f^{k+1}}\cdot\delta\bullet(f)f^s
\end{eqnarray}
for each $g(x,s)$ in the stalk $\calO_{X,\x}[s]$
and each $\CC$-linear derivation
  $\delta$ on $\calO_{X,\x}[s]$. The left ideals
\begin{eqnarray*}
\ann_{\calD_X[s]}(f^s)&=&\{P\in \calD_X[s]\mid P\bullet f^s=0\},\\
\ann_{\calD_X}(f^s)&=&\calD_X\cap\ann_{\calD_X[s]}(f^s)
\end{eqnarray*}
contain important information on the singularity structure of $f$.
\item The Bernstein--Sato polynomial $b_f(s)$: for $X=\CC^n$, this is
  the (monic) generator of the $\CC[s]$-ideal consisting of all
  polynomials $b_P(s)$ appearing in a functional equation of the form 
\begin{eqnarray}\label{eqn-bfu} 
(P\cdot  f)\bullet f^s&=&b_P(s)\cdot f^s
\end{eqnarray} 
where $P\in\calD_X[s]$. A theorem of Bernstein asserts
  that $b_f(s)$ is nontrivial, so that the root set
\[
\rho_{f}=\{\alpha\in\CC\mid b_f(\alpha)=0\}
\]
is finite, \cite{Bernstein-Nauk72}. 
For $\x\in X$ there are local and analytic versions $b_{f,\x}(s)$
with coefficients in the corresponding local or convergent power
series rings, and $b_f(s)=\lcm_{\x\in X}b_{f,\x}(s)$, compare \cite{NarvaezMebkhout-AnnEcNormSup91}.
\item The Milnor fiber $M_{f,\x}$ at a point $\x\in\Var(f)$: let
  $B(\x,\eps)$ denote the $\eps$-ball around $\x\in\Var(f)$. Milnor proved
  that the diffeomorphism type $M_{f,\x}$ of the open real manifold 
$M_{\x,t,\eps}=B(\x,\eps)\cap \Var(f-t)$
is independent of $\eps,t$ as long as $0<|t|\ll\eps\ll 1$. For
$0<\tau\ll\eps\ll 1$ there is a fiber bundle
\[
M_{f,\x}\into
B(\x,\eps)\cap \{\frakq\in\CC^n\mid
0<|f(\frakq)|<\tau\}\to B(0_{\CC^1},\tau)\smallsetminus \{0_{\CC^1}\}.
\] 
\item Monodromy at $\x\in\Var(f)$: the above fibration induces for all
  $k\in\NN$ smooth
  vector bundles 
  $H^k(M_{\x,t,\eps},\CC)$ over the base of the fibration.
The linear transformation $\mu_{f,\x}$ induced on
  $H^\bullet(M_{f,\x})$ by lifting the path $t\leadsto t\cdot
  \exp(2\pi \sqrt{-1} \lambda)$ for $\lambda\in[0,1]$ is the monodromy
  of $f$ at $\x$. We denote $e_{f,\x}=\{\gamma\in\CC\mid \gamma\text{
    is an eigenvalue of }\mu_{f,\x}\}.$
\end{enumerate}

The following are classical results on these invariants:
\begin{itemize}
\item[(i)] $
\bigcup_{\x\text{ near }\x_0}e_{f,\x}=\exp(2\pi \sqrt{-1} \rho_{f,\x_0})$, \cite{Malgrange-isolee,Kashiwara-vanishing83};
\item[(ii)] 
$\rho_{f,\x}\subseteq\QQ\cap (-n,0)$, \cite{Kashiwara-bfu, Saito-Bull94};
\end{itemize}

\subsubsection*{Logarithmic forms and $\ann_{\calD_X}(f^s)$}
In this note we prove  structural results for
$\ann_{\calD_X}(f^s)$ in the
presence of suitable homogeneities of $f\in\calO_X$. 
If there is a vector field $E$ with $E\bullet(f)=f$ then $f$ is
\emph{Euler-homogeneous}, cf.~Definition~\ref{dfn-Euler} for details
and strengthenings.
Our main tool are the sheaves
$\Omega^\bullet_X(\log f)$ of \emph{logarithmic
differential forms} along $f$. If $\Omega^1_X$ is the sheaf 
of
$\CC$-linear differentials on $X=\CC^n$, and setting
$\Omega^i_X=\bigwedge^i\Omega^1_X$, then  
$\Omega^i(\log f)$ is the (reflexive) sheaf of
differential forms $\omega$, meromorphic along $\Var(f)$, such that
$f\omega$ and $f\de(\omega)$ are holomorphic:
\begin{eqnarray*}
\Omega^i_X(\log f)=\{\omega\in\frac{1}{f}\Omega^i_X\mid
\de f\wedge\omega\in \Omega^{i+1}_X\}
=\{\omega\in\frac{1}{f}\Omega^i_X\mid 
   \de(\omega)\in\frac{1}{f}\Omega^{i+1}_X\}.
\end{eqnarray*}
This construction does not depend on the choice of the coordinate
system, and if $f=gu$ for a local unit $u$ then $\Omega^i_X(\log
f)=\Omega^i_X(\log g)$.  Thus, given an effective divisor $Y\subseteq
X$, $X$ not necessarily $\CC^n$, one can define a sheaf
$\Omega^i_X(\log Y)$ locally as $\Omega_U^i(\log f)$ for any local
defining equation $f$ for $Y$ on the open affine set $U$ in $X$. For
varieties over $\CC$, the construction is compatible with the
algebraic-analytic comparison map. 

The logarithmic forms of order $n-1$ induce a decomposition of $X$
into disjoint locally closed sets, see Definition
\ref{dfn-log-strat}. If this is a (locally finite) stratification, $f$
is called \emph{Saito-holonomic}, cf.~Definition~\ref{dfn-holonomic}.
In the same article \cite{Saito-logarithmic} where K.~Saito suggested
holonomicity, he also introduced the notion of \emph{freeness},
characterized by $\Omega^1_X(\log f)$ being a locally free
$\calO_X$-module. Free divisors have many nice properties and many
distinguished classes of divisors are free; this includes Coxeter
arrangements, discriminants in certain prehomogeneous vector spaces,
and discriminants in the base of a versal deformation of isolated
(complete intersection) singularities.  Calderon-Moreno and
Narvaez-Macarro \cite{Calderon-AnnSciEN99,Narvaez-Contemp08} studied
certain free divisors from the differential point of view. Part of
this note is a generalization and sharpening of their results to a
significantly larger class of divisors,
the central property being tameness, compare Definition
\ref{dfn-tame}. Tame divisors include all free divisors, and
all divisors in dimension three or less.

Since exterior products of logarithmic
differential forms are logarithmic again, $\Omega^\bullet_X(\log f)$ is a
complex of $\calO_X$-modules with differential $\dff\wedge$, and one 
can define $\calO_X$-submodules
\begin{eqnarray}\label{Omega0}
\Omega^i_{X}(\log_0 f)&:=&\ker\left(\de f\wedge(-)\colon \Omega^i_X(\log f)\to
\Omega^{i+1}_X(\log f)\right)
\end{eqnarray}
which {\em do depend} on the specific choice of the defining equation
$f$ for $\Div(f)$.  




\subsubsection*{\Logarithmic complexes}
Let $\pi\colon T^*X\to X$ be the canonical projection from the total
space of the cotangent bundle on $X$. The wedge product with the
Liouville form $\ydx$ (see Subsection \ref{subsec-complex} below)
defines a $\pi_*\pi^*(\calO_X)$-morphism 
\[
\ydx \colon \pi_*\pi^*(\Omega^{i-1}_{X}(\log_0 f))\to
\pi_*\pi^*(\Omega^i_{X}(\log_0 f))
\]
for all $i$.  In this note we construct from this morphism the
\emph{\logarithmic complex} $\C_f^\bullet$ of $f$, a global version of
a certain approximation complex from \cite{HerzogSimisVasconcelos}; it
is the main object of study in Section \ref{sec-complex}.  The
terminal cohomology group of the complex $\C_f^\bullet$ of (reflexive)
$\pi_*\pi^*(\calO_X)$-modules is naturally identified with the
quotient of $\pi_*\pi^*(\calO_X)$ by the \emph{\logarithmic ideal} of
$f$, denoted $\calL_f$.
If $Y=\Div(f)$ is a strongly Euler-homogeneous (effective) divisor
(Def.~\ref{dfn-Euler}) then the \logarithmic ideal $\calL_f$ is
independent of the choice of $f$ and only depends on $Y$, cf.~Remark
\ref{rmk-Ehom}.

Our main result on the \logarithmic complex is the following theorem;
see Section \ref{sec-complex} for details.
\begin{thm}
If $f\in \calO_X$ is tame, Saito-holonomic and strongly
Euler-homogeneous (but not necessarily reduced) then the \logarithmic complex
is a  resolution (of reflexive modules) for
the \logarithmic ideal $\calL_f$, and $\calL_f$ is a
Cohen--Macaulay prime ideal of dimension $n+1$.
\end{thm}
One consequence of the above theorem is that under the stated
hypotheses the symmetric algebra of the Jacobian ideal agrees with the
Rees algebra of $f$, and both have a linear resolution, Corollary
\ref{cor-blowup}. In a different direction one obtains information on
differential invariants of $f$ as the \logarithmic ideal supplies a
link between $\Omega^\bullet_X(\log f)$ and $\ann_{\calD_X}(f^s)$:
\begin{thm}
Suppose $f\in \calO_{X,\x}$ is tame, Saito-holonomic, strongly
Euler-homo\-geneous with strong
Euler field $E_\x$ at $\x\in X$,  but is not necessarily reduced.  Then
$\ann_{\calD_{X,\x}}(f^s)$ 
is generated by 
the $\CC$-linear derivations $\delta\colon \calO_{X,\x}\to
\calO_{X,\x}$ for which $\delta\bullet(f)=0$.
In consequence,
$\ann_{\calD_{X,\x}[s]}(f^s)=\calD_{X,\x}[s]\cdot(\ann_{\calD_{X,\x}}(f^s),E_\x-s)$
is generated by order one operators. 

\end{thm}

\subsubsection*{Hyperplane arrangements}

Let $X=\CC^n$ and denote by $D_n$ the ring of $\CC$-linear algebraic
differential operators on $X$ (\emph{i.e.}, the $n$-th complex Weyl
algebra).  A very interesting class of divisors are hyperplane
arrangements $\calA$, defined by a product $f_\calA$ of linear
polynomials. In generalization to the results in \cite{Terao-JAlg02}
which discuss differential operators with constant coefficients, Terao
conjectured around 2002 that $\ann_{D_n}(1/f_\calA)$ be generated by
differential operators of order one whenever $f_\calA$ is
reduced. Corresponding speculations have been made about
$\ann_{D_n}(f_\calA^s)$ in
\cite{Torrelli-BullMathSoc04,W-Bernstein}. We use our techniques to
prove that for tame arrangements $\ann_{D_n}(f_\calA^s)$ is indeed
generated by derivations. In consequence, Terao's
conjecture must hold in the tame case, but using an approach via local
cohomology we actually prove it in Theorem \ref{thm-terao} for all
arrangements, irrespective of tameness or multiplicities. On the other
hand, we provide in Example~\ref{exa-bracelet} an arrangement for
which $\ann_{D_n}(f_\calA^s)$ is \emph{not} generated by derivations
(while, of course, $\ann_{D_n}(1/f_\calA)$ still \emph{is} generated
by operators of order one).

\smallskip

Aside from the fact that $-1$ is always a root of the Bernstein--Sato
polynomial $b_f(s)$ of a non-constant polynomial $f$, very little is
known about specific roots. Budur, \mustata\ and Teitler 
formulated the following idea:
\begin{cnj}[The $n/d$-conjecture, 
\cite{BudurMustataTeitler-GeomDed11}]\label{cnj-n/d}
If $f_\calA$ defines a central reduced indecomposable arrangement of
$d$ hyperplanes in $\CC^n$ then $-n/d$ is a root of the
Bernstein--Sato polynomial of $f_\calA$.
\end{cnj}
For any central indecomposable arrangement $\calA$ (tame or otherwise)
we prove in Theorem \ref{thm-n/d} that all derivations that kill
$f_\calA$ (or $1/f_\calA$, or $f_\calA^s$) lie in the ideal
$D_n\cdot(x_1,\ldots,x_n)$.  This in turn has the following
consequence.
\begin{thm}\label{thm-n/d-arr}
Let $f_\calA$ be a central, indecomposable, not necessarily reduced,
arrangement of degree $d$ in $n$ variables for which
$\ann_{D_n}(f_\calA^s)$ is generated by derivations; for example,
$\calA$ could be tame. Then, the $n/d$-conjecture holds for $f_\calA$.
\end{thm}

\medskip
The interest in the $n/d$-Conjecture comes from the Strong
(topological) Monodromy Conjecture. To state that, denote by $Z_f(s)$
the topological zeta function $Z_f(s)$ attached to a divisor
$\Div(f)$. This is the rational function
\begin{eqnarray}\label{eqn-zeta-fcn}
Z_f(s)&=&\sum_{I\subseteq S}\chi(E^*_I)\prod_{i\in
  I}\frac{1}{N_is+\nu_i}
\end{eqnarray}
where $\pi\colon (Y,\bigcup_S E_i)\to (\CC^n,\Var(f))$ is an embedded
resolution of singularities, and $N_i$ (resp.\ $\nu_i-1$) are the
multiplicities of $E_i$ in $\pi^*(f)$ (resp.\ in the Jacobian of
$\pi$). By results of Denef and Loeser in \cite{DenefLoeser-JAG98}, $Z_f(s)$
is independent of the resolution. The Strong (topological) Monodromy
Conjecture claims that any pole of $Z_f(s)$ is a root of the
Bernstein--Sato polynomial $b_f(s)$ (see \cite{Budur-survey} for
more information). The
Strong Monodromy Conjecture is a variant of an old (and still wide
open) conjecture of Igusa linking $p$-adic integrals to the root set
$\rho_f$, cf.\ \cite{Denef-Bourbaki}.

Theorem \ref{thm-n/d-arr} then implies, by ideas of Budur,
\mustata\  and Teitler from \cite{BudurMustataTeitler-GeomDed11}:
\begin{cor}
The Strong Monodromy Conjecture holds for all tame arrangements
(irrespective of centrality, indecomposability or reducedness). 
\end{cor}
This result extends in a new direction some of the results in
\cite{BudurSaitoYuzvinsky-JLMS11} as every arrangement in dimension
three or less is tame.  It proves the Strong Monodromy Conjecture for
all arrangements of the following types: Coxeter arrangements;
Ziegler's multi-reflection arrangements; arrangements of
crystallographic type; all multi-arrangements in dimension 3.

\medskip

A folklore conjecture states that the Bernstein--Sato polynomial of a
hyperplane arrangement $\calA$ depends only on the intersection
lattice. Inspired by Theorem \ref{thm-H-milnor} discussed below we
construct in Section \ref{sec-Milnor} a pair of arrangements that
defeats this conjecture. However, we offer an improved conjecture
involving a slightly finer combinatorial invariant than the usual
intersection lattice. We describe next our approach towards Theorem
\ref{thm-H-milnor}.

\subsubsection*{Jacobian module and monodromy}

Let  $X=\CC^n$ and set $R_n=\CC[x_1,\ldots,x_n]$ with maximal
homogeneous ideal $\frakm$.
If $f$ has an isolated singularity at the origin, the Milnor fiber
$M_f$ at the origin is a bouquet of spheres, and in the
Euler-homogeneous case the number of these spheres equals the
$\CC$-dimension of the Jacobian ring $R_n/\Jac(f)$. Malgrange showed
that for Euler-homogeneous isolated singularities $R_n/\Jac(f)$ has a $\QQ[s]$-module
structure where $s$ acts via the Euler-homogeneity,
\cite{Malgrange-isolee}. If $f$ is, in addition,  quasi-homogeneous then
the root set of $b_f(s)/(s+1)$ is in bijection with the degree set of
the nonzero quasi-homogeneous elements in $R_n/\Jac(f)$. For
positive-dimensional singular loci, much of this breaks down, since
$R_n/\Jac(f)$ is not Artinian in that case. Let the \emph{Jacobian
  module} be
\[
H^0_\frakm(R_n/\Jac(f))=\{g+\Jac(f)\mid 
\exists k\in\NN, \forall i\,\, x_i^kg\in
\Jac(f)
\}.
\]
We
give in Section \ref{sec-Milnor} the following generalization of
Malgrange's result 
which has recently been used by A.~Dimca and G.~Sticlaru in
\cite{DimcaSticlaru}.

\begin{thm}\label{thm-H-milnor}
Let $f\in R_n$ be reduced and homogeneous of degree $d$, with $n\geq
2$. Assume that $\Proj(R_n/f)$ has isolated singularities.
Then,
with $1\le k\le d$ and $\lambda=\exp(2\pi \sqrt{-1}k/d)$,
\begin{eqnarray*}
\dim_\CC [H^0_\frakm(R_n/\Jac(f))]_{d-n+k}&\le& 
\dim_\CC \gr^{{\rm Hodge}}_{n-2}(H^{n-1}(M_f,\CC)_\lambda)
\end{eqnarray*}
where the right hand side indicates the $\lambda$-eigenspace of the
associated graded object to the Hodge filtration on $H^{n-1}(M_{f},\CC)$,
and where the left hand side denotes the graded component of the
Jacobian module.
\end{thm}


\section{Stratifications  and homogeneity}

We collect in this section useful facts about stratifications and
homogeneity properties.
As always, $X$ is a complex manifold with structure sheaf $\calO_{X}$ of
holomorphic functions. Points of $X$ are usually denoted $\x$ and the
stalk of an object at $\x$ by $(-)_\x$.

\subsection{Stratifications}

\subsubsection*{Whitney stratifications}

Whitney \cite{Whitney-Annals65} showed that all closed analytic spaces
$Y\subseteq X$ have a \emph{Whitney stratification}, satisfying
certain conditions on limits of tangent spaces of strata.  If
$Y'\subseteq Y$ are closed analytic, there exist Whitney
stratifications $\Sigma_{Y,Y'}$ of $Y$ such that each stratum is
either contained in or disjoint to $Y'$, and in addition the strata
inside $Y'$ form a Whitney stratification for $Y'$.

Each analytic space, except for finite sets of points, has infinitely many
Whitney stratifications. There is however, one distinguished such:
\begin{dfn}
The \emph{canonical Whitney stratification} $\Sigma_{Y}$ of the
analytic set $Y$ is defined inductively as follows:
\begin{itemize}
\item $W_0$ is the regular part of $Y$;
\item if $Y^k$ denotes $Y\minus\bigcup_{i=0}^{k-1} W_i$ then $W_k$ is
  for $k>0$ the points of the regular part $Y^k_\reg$ of $Y^k$ at which
  the
  Whitney conditions for the pairs $(W_i,Y^k_\reg)$ are satisfied for
  all $i<k$.
\end{itemize}
\end{dfn}

\subsubsection*{Logarithmic stratification}
If $\delta$ is an analytic vector field on $X$ then
integrating $\delta$ leads to a foliation of $X$ into curves near any
point where $\delta$ does not vanish. The following definition was
made by K.~Saito for divisors $Y$.
\begin{dfn}\label{dfn-log-fields}
For an ideal sheaf $\calI_Y\subseteq\calO_X$, let
$\Der_X(-\log Y)$ be the sheaf generated locally by the vector fields
$\delta$ with $\delta\bullet(\calI_Y)\subseteq \calI_Y$.
\end{dfn} 
It follows from the product rule that
logarithmic derivations are indeed a function of the ideal and not of
a set of chosen generators.
If
$f$ cuts out $Y$ we use $\Der_X(-\log Y)$ and $\Der_X(-\log f)$
interchangeably, but we distinguish between logarithmic data along $Y$
and its reduced scheme $Y_\red$. 

For a reduced divisor $Y=Y_\red$, K.~Saito \cite{Saito-logarithmic}
introduced logarithmic derivations, and also the following concept.
\begin{dfn} \label{dfn-log-strat}
Pick $\x,\x'\in X$ and write $\x\sim_Y \x'$ if there exist an open
set $U\subseteq X$ containing $\x,\x'$ and a derivation
$\delta\in\Der_U(-\log (Y\cap U))$, 
vanishing nowhere on $U$, such
that one of the integrating curves of $\delta$ passes through both
$\x$ and $\x'$.

Varying over all open $U\subseteq X$ and all $\delta \in\Der_U(-\log
(Y\cap U))$ one induces an equivalence relation denoted $\x \approx_Y
\x'$.  The strata of the {\em logarithmic stratification}
$\frakS_{X,Y}$ of $X$ induced by $Y$ are then by definition the
irreducible components of the cosets of the relation $\approx_Y$.
The set of strata includes $X\minus Y$, and
the components of the non-singular locus  of $Y_{\text{red}}$. 
\end{dfn}

Saito noted cases where this stratification is not locally finite: the
dimension of $\{\x\in X\mid \rk_\CC(\Der_X(-\log Y)\otimes
\calO_{X,\x}/\frakm_\x)=i\}$ can be greater than $i$.
\begin{exa}\label{exa-saito}
Let $Y=\Var(xy(x+y)(x+zy))\subseteq X=\CC^3$. Then $\Der_{X}(-\log Y)$
vanishes identically on the $z$-axis and the $z$-axis is an
irreducible component of $\{\x\in X\mid \rk_\CC(\Der_X(-\log
Y)\otimes_{\calO_X} \calO_{X,\x}/\frakm_\x)=0\}$.  Each point on the
$z$-axis is its own logarithmic stratum.  
\schluss
\end{exa}
\begin{dfn}[\cite{Saito-logarithmic}]\label{dfn-holonomic}
If the logarithmic stratification induced by $Y$ on $X$ is everywhere
locally finite then  $Y$ is called \emph{Saito-holonomic}.
\end{dfn}
From now on, $Y$ will be a divisor, not necessarily reduced.
\begin{rmk}\label{rmk-Whitney-diffeo-product}
\begin{asparaenum}
\item\label{rmk-Whitney-product} Let $\sigma$ be a
  positive-dimensional stratum in a Whitney stratification for
  $\Var(f)=Y\subseteq \CC^n$, and let $\x$ be a point in $\sigma$.
  Thom and Mather have shown that near $\x$ there is a homeomorphism
  of germs between $(\CC^n,\Var(f))$ and $(\CC\times
  \CC^{n-1},\CC\times \Var(g))$ where $g=g(x_2,\ldots,x_n)$ for
  suitable $g$.
\item The homeomorphisms of the previous item can in general not be
  chosen differentiably: the divisors to $xy(x+y)(x+2y)$ and
  $xy(x+y)(x+zy)$ in $\CC^3$ are locally topologically equivalent
  outside $z(z-1)(z-2)=0$. Yet, a differentiable isomorphism of the
  germs would induce a correspondence of their logarithmic
  stratifications, which is manifestly impossible along the $z$-axis.
\item Let $Y\subseteq X=\CC^n$ be Saito-holonomic and let
  $\x\in\sigma$ be a point in a logarithmic stratum. Then the
  evaluation of $\Der_{X,\x}(-\log Y)$ at $\x$ spans the tangent space
  of $\sigma$ at $\x$. Note that in Example \ref{exa-saito}, the
  points where the logarithmic vector fields all vanish forms a
  positive-dimensional set, corresponding to a locally infinite
  stratification. The loci $\{\x\in X\mid \rk_\CC(\Der_X(-\log
  Y)\otimes_{\calO_X} \calO_{X,\x}/\frakm_\x)=i\}$ are the candidates
  for logarithmic strata.  They can be computed via Fitting ideals and
  are certified as holonomic strata if they have the expected
  dimension.
\item \label{rmk-Whitney-splitting} Suppose $Y$ is
  Saito-holonomic. Near $\x\in\sigma$, consider the foliation of $X$
  obtained by integrating $\dim(\sigma)$ many independent elements
  $\delta_i$ in $\Der_X(-\log f)$ with $\delta_i\bullet(f)=0$ that do
  not vanish near $\x$ (any such collection is involutive). A choice
  of $n-\dim(\sigma)$ independent vector fields that are transversal
  to $\sigma$ at $\x$ induces a local analytic isomorphism between the
  pair $(X,Y)$ and the product of $\CC^{\dim\sigma}$ with a pair
  $(\CC^{n-\dim(\sigma)}=X',Y')$ where the divisor $Y'\subseteq X'$ is
  a cross-section of $Y$ transversal to $\sigma$ at $\x$.
  Compare \cite[3.5, 3.6]{Saito-logarithmic}.
\item Whitney shows in \cite{Whitney-Annals65} that the canonical
  Whitney stratification $\Sigma_{X,Y}$ is stable under all local
  analytic isomorphisms of $X$ that fix $Y$.  In the Saito-holonomic case,
  the local product structure implies that the logarithmic
  stratification refines the canonical Whitney stratification and is
  Whitney itself, compare \cite[Prop.~3.11]{DamonMond} and
  \cite{DamonPike} for further details.
\item To any ideal $\calI_Y$ in $\calO_X$ one can associate the (left)
  ideal in the sheaf of $\CC$-linear differential operators $\calD_X$
  on $X$ generated locally by the logarithmic derivations along
  $Y$. The cokernel $M^{\log Y}$ of this ideal may or may not be
  holonomic in the sense of Kashiwara, even for divisors. The notions
  of holonomicity are not the same: while Saito-holonomicity implies
  that $M^{\log Y}$ is a holonomic module in the free case 
  \cite[(3.18)]{Saito-logarithmic}, the example $(x-yz)(x^4+x^4y+y^4)$
  shows that the implication cannot be reversed.
\schluss
\end{asparaenum}
\end{rmk}


\subsection{Homogeneity conditions}

\begin{dfn}\label{dfn-Euler}
Let $\x\in X$.
\begin{itemize}
\item $f\in\calO_{X,\x}$ is \emph{Euler-homogeneous} at $\x$
  if there is a vector field $E_\x$ on
  $\calO_{X,\x}$ 
  with $E_\x\bullet(f)=f$.  If $E_\x$ can be chosen
  to vanish at
  $\x$ then $f$ is called \emph{strongly Euler-homogeneous} at $\x$.
\item $f\in\calO_X$ is (strongly) Euler-homogeneous if it is so at
  each $\x\in \Var(f)$.
\item A (not necessarily reduced) divisor $Y\subseteq X$ is (strongly)
  Euler-homogeneous, if near every $\x\in Y$ there is a (strongly)
  Euler-homogeneous germ $f_\x\in\calO_{X,\x}$ with $Y_\x=\Div(f_\x)$.
\end{itemize}
\end{dfn}

\begin{rmk}\label{rmk-Ehom}
\begin{asparaenum}
\item Any $f\in\calO_X$ is strongly Euler-homogeneous in every
  smooth point of the associated reduced divisor.
\item If $u\in\calO_U$ is a
  unit then Euler-homogeneity is not always inherited from $f$ to $uf$:
  $\exp(z)\cdot f(x,y)$ is Euler-homogeneous via $\del_z$, but $f$ has
  no reason to be so as well.
\item In contrast, $f$ is strongly Euler-homogeneous on $U$ if and
  only $uf$ is. Indeed, if $E_\x$ is a strong Euler field for $f$ then
  $\frac{uE_\x}{u+E_\x\bullet(u)}$ is one for $uf$.
\item \label{rmk-Ehom-item-product} Suppose $Y\subseteq X$ is a
  divisor. If there is an analytic splitting $(X,Y)=\CC\times (X',Y')$
  then $Y$ is strongly Euler-homogeneous if and only if $Y'$ is,
  \cite[Lem.~3.2]{GrangerSchulze-Compos06}.  On the other hand,
  Euler-homogeneity is not necessarily inherited from $Y$ to $Y'$:
  $\CC\times Y'$ is always Euler-homogeneous thanks to the exponential
  function.
\item Euler-homogeneity of $f$ is an open condition, while strong
  Euler-homogeneity is not. For example,
  $f=zx^4+xy^4+y^5$ is strongly Euler-homogeneous at the origin, but
  it cannot be so along the $z$-axis because otherwise $f$
  should be in the ideal $(x,y,z-\lambda)\cdot(f_x,f_y,f_z)$, which it
  is not unless $\lambda=0$.
\item A strongly Euler-homogeneous divisor may fail to be Saito
  holonomic: consider $f=xy(x+y)(x+zy)$. Indeed, along the $z$-axis, $f$
  has the strong Euler derivation $(x\del_x+y\del_y)/4$; along the
  $y$-axis and along the line $x+y=z-1=0$, $f$
  has the strong Euler derivation $(x+zy)\del_z$; 
  in all other places 
  $\Var(f)$ is smooth. On the other hand, every point on the $z$-axis
  is a logarithmic stratum.
\item \label{rmk-splitting-strE} If $f$ is strongly Euler-homogeneous
  and Saito-holonomic at $\x$, $\x$ a point in the stratum $\sigma$,
  \ref{rmk-Whitney-diffeo-product}.\eqref{rmk-Whitney-splitting}
  yields a splitting where $f$ is constant along $\dim(\sigma)$
  directions of $X'$, and the transversal section is strongly
  Euler-homogeneous.
\schluss
\end{asparaenum}
\end{rmk}

\begin{dfn}\label{dfn-Der0}
For $f\in\calO_X$ we set
\[
\Der_{X}(-\log_0 f):=\Der_X(-\log
f)\cap\ann_{\calD_X}(f^s)=\{\delta\in\Der_X(-\log f)\mid
\delta\bullet(f)=0\}.
\]
\end{dfn}
\begin{rmk}\label{rmk-Der0}
\begin{asparaenum}
\item Geometrically, elements of 
  $\Der_X(-\log f)$ are tangent to the hypersurface $\Var(f)$ while
  those in $\Der_{X}(-\log_0 f)$ are tangent to all level hypersurfaces of
  $f$.
\item Let $x, x'$ be two local coordinate systems at $\x\in X$. Then
  the gradients $\nabla_x(f)$ and $\nabla_{x'}(f)$ differ by the
  Jacobian matrix: $\nabla_x(f)=(\nabla_{x'}(f))\cdot ((\frac{\del
    x'}{\del x}))$. Hence, $\Der_{X}(-\log_0 f)$ varies in a dual
  fashion with the coordinate system.
\item If $u$ is a local unit, $E_\x$ a strong Euler-homogeneity for
  $f\in\calO_{X,\x}$ and $\delta\in\Der_{X,\x}(-\log_0 f)$ then
  $\delta-\frac{\delta\bullet(u)}{u+E_\x\bullet(u)}E_\x\in\Der_{X,\x}(-\log_0
  (uf))$; this association is an $\calO_{X,\x}$-module
  isomorphism.
\item If $f$ is Euler-homogeneous at $\x$, then 
  the maps
\begin{gather}
\Der_{X,\x}(-\log f)\ni \delta\mapsto \delta-\frac{\delta\bullet(f)}{f}
\cdot E_\x\in \Der_{X,\x}(-\log_0 f),\\ 
\Der_{X,\x}(-\log f)\ni \delta\mapsto \frac{\delta\bullet(f)}{f} 
\cdot E_\x\in \calO_{X,\x}\cdot E_\x
\end{gather}
give a split exact sequence
\[
0\to\Der_{X,\x}(-\log_0 f)\to \Der_{X,\x}(-\log f)\to \calO_{X,\x}\cdot E_\x\to 0.
\]
These splittings depend on the choice of the equation $f$ for 
the divisor.
\schluss
\end{asparaenum}
\end{rmk}
\begin{rmk}
If $Y\subseteq X$ are algebraic, then logarithmic vector fields and
differentials, as well as differential annihilators can be defined
in both the analytic and the algebraic category.  Since they are
all defined by syzygies between derivatives of $f$, the analytic
objects are the pullbacks to the analytic category of the algebraic
objects. Questions about their generation and homological properties
can hence be investigated on either side.\schluss
\end{rmk}

\section{The \logarithmic complex}\label{sec-complex}



As always, $X$ is a complex analytic manifold. 
Let $f$ be a global section of $\calO_X$.
We consider the principal $\calD_X[s]$-module $\calD_X[s]\bullet f^s$
generated by the symbol $f^s$ subject to local relations spelled out in
\eqref{eqn-intro-fs}. 
Whenever $f$ is Euler-homogeneous, $\calD_X[s]\bullet f^s$
  and $\calD_X\bullet f^s$ agree.
In this section we determine the structure of $\ann_{\calD_X[s]}(f^s)$
for a large family of divisors with good homogeneity conditions by
computing its characteristic cycle.

\subsection*{The \logarithmic ideal}

For any filtration $F$ on a ring $A$ we denote by $\gr_F(A)$ the
associated graded object, and by $\gr_F(a)$ the image of $a\in A$ in
$\gr_F(A)$.  The order filtration on $\calD_X$ leads to a sheaf of
$(0,1)$-graded commutative rings $\gr_{(0,1)}(\calD_X)$
whose sections are naturally identified with the functions on $T^*X$
that are polynomial in the cotangent directions.


\begin{dfn}\label{dfn-Lf}
Let the \emph{\logarithmic ideal} be the ideal $\calL_f\subseteq
\gr_{(0,1)}(\calD_X)$ generated by the $(0,1)$-symbols (that is, the
lead terms under the order filtration) of $\Der_{X}(-\log_0
f)\subseteq\ann_{\calD_X[s]}(f^s)$:
\[
\calL_f=\gr_{(0,1)}(\calD_X)\cdot \gr_{(0,1)}(\Der_{X}(-\log_0 f)).
\]
\end{dfn}

\begin{rmk}\label{rmk-Lf2}

Suppose $f$ is Euler-homogeneous at $\x\in X$: $E_\x(f)=f$. In this
remark, we read derivations as formal $\calO_X$-linear combinations on
the derivatives $f_i=\del_i\bullet(f)$. In this sense, elements of
$\Der_X(-\log_0 f)$ are syzygies.
\begin{asparaenum}
\item Let $x,x'$ be two coordinate systems with (column) vectors of
  partial differentiation operators $\del,\del'$. Denote
  $c_\delta,c'_\delta$ the coefficient (column) vectors of a
  derivation $\delta=c_\delta^T\cdot\del={c'_\delta}^T\cdot\del'$ in
  the two coordinate systems. If $J=((\del x'_j/\del x_i))$ is the
  Jacobian matrix, $\del=J\cdot\del'$ and $c'_\delta=J^T\cdot
  c_\delta$. In particular, $\Der_X(-\log_0 f)=J^{_T}\cdot
  \Der'_X(-\log_0f)$ and so $\calL_f$ is well-defined.

\item Let $u\in\calO_{X,\x}$ be a unit. The $\calO_{X,\x}[s]$-algebra
  automorphism $\del_i\to \del_i-\frac{\del_i\bullet(u)}{u}s$ on
  $\calD_{X,\x}[s]$ identifies $\ann_{\calD_{X,\x}[s]}(f^s)$ with
  $\ann_{\calD_{X,\x}[s]}(u^sf^s)$. On the level of derivations, if
  there is a strong Euler field $E_\x$ for $f$ at $\x$, this
  corresponds to an $\calO_{X,\x}$-isomorphism $\alpha_u\colon
  \Der_{X,\x}(-\log_0 f)\to \Der_{X,\x}(-\log_0(uf))$ via
  $\delta\mapsto\delta-\frac{\delta\bullet(u)}{u+E_\x\bullet(u)}E_\x$ from
  Remark \ref{rmk-Der0}.
\item If $\nabla(u),\nabla'(u)$ are the gradient (column) vectors of
  $u$ in the two coordinate systems, $\nabla(u)=J\cdot\nabla'(u)$.  The
  $\calO_{X,\x}$-isomorphism $\alpha_u$ from the previous item  sends
  $c_\delta^T\cdot\del$ to $c_\delta^T\cdot(I_{n}-\frac{\nabla(u)\cdot
  c_{E_\x}^T} {u+E_\x\bullet(u)})\cdot\del$, where $I_n$ is the
  identity matrix.  On the other hand, in $x'$-coordinates,
  $c_\delta^T\cdot\del=c_\delta^T\cdot J\cdot \del'$ is sent to
  $(c_\delta^T\cdot J)(I_{n}-\frac{\nabla'(u)\cdot
    {c'}_{E_\x}^T}{u+E_\x\bullet(u)})\del'$. A simple calculation
  shows now that $\alpha_u$ commutes with coordinate changes.
\item One may consider the symbols $y_f=\gr_{(0,1)}(\del)$ as
  indeterminates over $\calO_X$, labeled by the choice of the defining
  equation $f$ for a fixed strongly Euler-homogeneous
  divisor. Locally, the maps $\alpha_u$ allow to identify these
  symbols, respecting the action by the Jacobian under a coordinate
  change.  With $(y_{uf})^T(I_{n}-\frac{\nabla'(u)\cdot
    {c'_{E_\x}}^T}{u+E_\x\bullet(u)})=y_f^T$ there is a local
  $\calO_X$-isomorphism of $\gr_{(0,1)}(\calD_X)$ sending
  $\gr_{(0,1)}(\delta)\mapsto
  \gr_{(0,1)}(\delta)-\frac{\delta\bullet(u)}{u+E_\x\bullet(u)}\gr_{(0,1)}(E_\x)$
  that commutes with coordinate changes, patches over the domain where
  $E_\x$ and $u$ are defined, and sends $\calL_f$ to $\calL_{uf}$.

  It follows that if $Y$ is a strongly Euler-homogeneous divisor then
  the local geometric and algebraic properties of $\calL_Y$ are
  independent of the choice of the local equation cutting out $Y$.
\item For any coherent $\calD_X$-module $\calM$, filtered compatibly
  with the order filtration on $\calD_X$, the support in the cotangent
  bundle $T^*X$ of the associated graded object $\gr_{(0,1)}(\calM)$
  is the \emph{characteristic variety} $\charVar(\calM)$.

Generically (in $X$), $\calL_f$ is of height $n-1$ in
$\gr_{(0,1)}(\calD_X)$ since $\Der_{X}(-\log_0 f)$ is of rank
$n-1$. Wherever some derivative $f_i$ is nonzero, the ideal $\calL_f$
is cut out by $\{y_j-\frac{f_j}{f_i}y_i\}$. In fact, (see, e.g.,
\cite{KashiwaraBook}) over the smooth locus of $f$ the characteristic
variety of $\calD_X\bullet(f^s)$ is locally a smooth complete
intersection of codimension $n-1$, and cut out by $\calL_f$.
\schluss\end{asparaenum}
\end{rmk}

\subsection{The complex}\label{subsec-complex}

Any symplectic manifold $X$ admits a one-form on the cotangent bundle
whose differential is the symplectic form. Any such one-form, not
necessarily unique, is called a \emph{symplectic potential}.  Let $X$
be a complex analytic manifold; it then has a natural symplectic
structure. The \emph{Liouville form} on $X$ is the unique symplectic
potential on $X$ that in local (Darboux) coordinates takes the form
\[
\ydx:=y_1\de x_1+\ldots+ y_n\de x_n
\]
with base coordinates $x$ and corresponding cotangent coordinates $y$. 

As functions on the cotangent bundle fibers, the $y_i$ can be
identified with the symbols of $\del_i$ in the corresponding
coordinate system.  This makes it clear that $\ydx$ is independent of
the coordinate system, since $\del_i$ and $\de x_i$ vary dually with any
coordinate transformation. 

For any $\calO_X$-module $\calM$, write 
\[
\calM[y]:=\pi_*\pi^*\calM=\calM\otimes_{\calO_X}\pi_*(\calO_{T^*X})
\] 
where $\pi\colon T^*X\to X$ is the natural projection from the
cotangent bundle.

\begin{dfn}
Suppose $f$ is a global non-constant function on $X$.  The modules of
the \emph{\logarithmic complex $C^\bullet_f$ of $f$} are defined by
\[
\Omega^n_X\otimes_{\calO_X}\frac{1}{f}\calO_X\otimes_{\calO_X}
C^i_f=\Omega_{X}^i(\log_0 f)[y]\subseteq \frac{1}{f}\Omega^i_X[y],
\]
\emph{i.e.}, up to twist by $1/f$ they are the kernels of $\de f\wedge(-)$ on
$\left(\omega_X^{-1}\otimes_{\calO_X}\Omega^\bullet_X\right)[y]$ where
$\omega_X^{-1}$ is the inverse of the invertible sheaf $\Omega^n_X$. 
The differential of $C_f^\bullet$ is
given by the exterior product with the Liouville form $\ydx$.
\end{dfn}

\begin{rmk}\label{rmk-Cf}
\begin{asparaenum}
\item 
Since $\de f \wedge\ydx=-\ydx\wedge\de f$, $C^\bullet_f$ is indeed a
complex in the category of $\calO_X[y]$-modules.
\item Locally, $(C^\bullet_f,\ydx)$ can be identified with
  $(\Omega^\bullet_X(\log_0 f)[y],\ydx)$.  On $X=\CC^n$, with a chosen
  coordinate system $x$ and corresponding derivations $\del$, identify
  $y$ with the order symbols of $\del$, and $\gr_{(0,1)}(\calD_X)$
  with $\calO_X[y]$. Denote by $K^\bullet$ the cohomological Koszul
  complex of the regular sequence $y$ on $\frac{1}{f}\calO_X[y]$ and
  identify $K^i= \frac{1}{f}\bigwedge^{n-i}(\calO_X[y])^n $ with
  $\frac{1}{f}\Omega^i_X[y]$ in the usual way. Then $C^\bullet_f$ is
  isomorphic to the restriction of $K^\bullet$ to the complex whose
  modules are the kernels in the Koszul complex on $\calO_X[y]$
  induced by $f_1=\del_1(f),\ldots,f_n=\del_n(f)$. As such, locally
  in a chosen coordinate system, the Liouville complex is 
  isomorphic to the approximation complex $\calC$ from
  \cite{HerzogSimisVasconcelos} and we use some of their techniques. 
\item \label{rmk-Cf-grading}
  The induced $y$-grading on $\gr_{(0,1)}(\calD_X)\cong
  \calO_X[y]$ makes $C^\bullet_f$ graded if we position
  $\calO_X\subseteq C^n_f$ in degree zero, and more generally the
  elements of $f\cdot\omega_X^{-1}\otimes_{\calO_X}\Omega^i_X(\log
  f)\subseteq f\cdot\omega_X^{-1}\otimes_{\calO_X}\Omega^i_X(\log
  f)[y]=C^i_f$ in degree $n-i$. This agrees with the intrinsic
  $y$-grading from the definition.

  We write $C^i_{f,\x}$ for the stalk of $C^i_f$ at $\x\in X$ (that
  is, we tensor with $\calO_{X,\x}$) and denote the part of $C^i_f$ in
  degree $j$ by $(C^i_f)_j$,
  \[
  C^i_f=\bigoplus (C^i_{f})_j.
  \]
  We have
  $\pdim_{\calO_{X,\x}[y]}(C^i_{f,\x})=\pdim_{\calO_{X,\x}}
  (C^i_{f,\x})_j$ whenever $(C^i_{f,\x})_j\neq
  0$ since 
\[
C^i_f=(C^i_{f})_{n-i}\otimes_{\calO_X}  \gr_{(0,1)}(\calD_X).
\] 
\item \label{rmk-Cf-syz}
In each stalk on $X$, the module $(C^i_f)_{n-i}$ is, up to a free
summand, a second syzygy (of the cokernel of $\de f\wedge$). 
By \cite[Thm.~3.6]{EvansGriffith-Syzygies}, the property of
$(C^i_f)_{n-i}$ of being a second syzygy over $\calO_X$ is equivalent
to it being $\calO_X$-reflexive since a regular ring is a normal
domain. Being a second syzygy also forces the $\calO_X$-projective
dimension of $(C^i_f)_{n-i}$ to be at most $n-2$.  Both properties
percolate to $C^i_f$.

  \item \label{rmk-Cf-fields-diffs}
  The isomorphism
  \[
  \Omega^n\otimes_{\calO_X}\frac{1}{f}\calO_X\otimes_{\calO_X}\Der_{X}(-\log_0
  f)\ni \frac{\de x}{f}\otimes \sum_{i=1}^n a_i\del_i\longleftrightarrow
  \frac{\sum_{i=1}^n(-1)^ia_i\widehat{\de
    x_i}}{f}\in\Omega^{n-1}_{X}(\log_0 f)
  \]
  shows that $\calO_X[y]/\calL_f$ is the terminal cohomology group of
  $C^\bullet_f$.
\item \label{rmk-Cf-dual} In general, one has $\Omega^i_X(\log f)=
  \left(\bigwedge^i(\Der_X(-\log f))^*\right)^{**}$, where the star
  denotes $\Hom_{\calO_X}(-,\Omega^n_X(\log f))$. If $f$ is strongly
  Euler-homogeneous (or more generally if $E$ is an Euler field for
  $f$ with $u+E\bullet(u)$ nonzero), Remark~\ref{rmk-Der0} implies
  that there is an $\calO_X[y]$-automorphism of $\gr_{(0,1)}(\calD_X)$
  that carries $\Omega^{i}_{X}(\log_0(f))[y]$ into
  $\Omega^i_{X}(\log_0(uf))[y]$. In particular, $C^\bullet_f$ and
  $C^\bullet_{uf}$ have the same algebraic properties.
  \schluss\end{asparaenum}
\end{rmk}

\subsection{Exactness}\label{subsec-exact}

\begin{ntn}
Throughout we use the word \emph{resolution} to denote a finite complex with
a unique cohomology group (at its end). No specific properties of the
modules involved (such as freeness, projectivity or injectivity) are
implied.
\end{ntn}

We show here that $C^\bullet_f$ is a resolution of $\calL_f$ when
$n\le 3$ or if the modules $\Omega^i_X(\log f)$ have locally high
depth over $\calO_X$.  For this, we need to slightly generalize the
construction of $C^\bullet_f$.
\begin{ntn}
Let $\phi=\phi_1,\ldots,\phi_n$ be in $\calO_X$, $X$ smooth and
affine.  Then let $C^\bullet_\phi$ be the (graded) subcomplex of
$(\Omega^\bullet_X[y],\ydx\wedge)$ given by the kernels of
$\phi\wedge(-)$, where we read $\phi$ as the coefficients of a
one-form. One can, as for $C^\bullet_f$, ask whether it is a
resolution.
\end{ntn}

\begin{lem}\label{lem-phi}
The complexes $C^\bullet_\phi$ and $C^\bullet_{g\phi}$ agree for all
nonzero $g\in\calO_X$. If $\phi$ generates an ideal of height at least
two at $\x\in X$ then $C^\bullet_\phi$ is exact at $\x$ for all
$i\le\min(\hight_\x(\phi),n-1)$ where $\hight_\x(\phi)$ is the 
height of the ideal generated
by $\phi$ in $\calO_{X,\x}$.
\end{lem}
\begin{proof}
The first part is clear as we calculate in a domain.

For the second part,  suppose
$\beta\wedge\phi=\beta\wedge \ydx=0$, $\beta\in\Omega^i_X[y]$,
$i\le n-1$.  Then $\beta=\ydx\wedge \gamma$ for
$\gamma\in\Omega^{i-1}_X[y]$.  Then $\beta':=\gamma\wedge\phi\in
C^i_\phi$ and $\beta'\wedge\ydx=0$.  One checks that this induces an
$\calO_X[y]$-endomorphism of $H^i(C^\bullet_\phi)$ sending the class
of $\beta$ to the class of $\beta'$ and which we denote $(-)_\phi'$.

Now $H^i(C^\bullet_\phi)$ is graded in $y$ and $(-)'_\phi$ reduces
degree by one. Thus, eventually $\beta^{(k)}:=(\beta^{(k-1)})'$ will
be the zero class and so $\beta^{(k)}=\gamma^{(k)}\wedge \ydx$ with
$\gamma^{(k)}\wedge \phi=0$. 
In a regular local ring,
$\hight_\x(\phi)$ is the largest $i$ such that the Koszul cocomplex of
$\phi$ on $\calO_X$ is exact in positions lower than $i$. 
Since $i\le \hight_\x(\phi)$,
$\gamma^{(k)}=\delta^{(k)}\wedge \phi$ for some $\delta^{(k)}\in
\Omega^{i-2}_X[y]$. 
Then
$-\delta^{(k)}\wedge \ydx+\gamma^{(k-1)}$ is in the kernel of $\phi$
and, like $\gamma^{(k-1)}$, multiplies against $\ydx$ to
$\beta^{(k-1)}$. So, if the class of $\beta^{(k)}$ in
$H^i(C^\bullet_\phi)$ is zero (that is, if $\beta^{(k)}$ is image under
$\wedge\ydx$ in $C^\bullet_\phi$) then the same is true for
$\beta^{(k-1)}$, and thus eventually of $\beta=\beta^{(0)}$.
\end{proof}

We make precise in the following definition the homological conditions
we require from our divisors; $U$ is an open set in $X$.

\begin{dfn}\label{dfn-tame}
The section $f\in \calO_U$ is \emph{tame} if for all $i$ the
projective dimension of $\Omega^i_U(\log f)$ over $\calO_U$ is at most
$i$ in each stalk. A divisor $Y$ is tame if it allows tame defining
equations locally everywhere.
\end{dfn}

\begin{thm}
\label{thm-C-exact}
For $f\in\Gamma(X,\calO_X)$, suppose that either
\begin{itemize}
\item $n\le 3$, or
\item $f$ is strongly Euler-homogeneous, 
Saito-holonomic, and tame.
\end{itemize}
Then the \logarithmic complex $C_f^\bullet$ is a resolution of the
\logarithmic ideal $\calL_f$.
\end{thm}

Before we embark on the proof, inspired by
\cite{HerzogSimisVasconcelos}, we collect some helpful facts.

\begin{ntn}
If $\delta$ is a vector field and $\omega$ a differential form on an
open set $U\subseteq X$ then we denote by $\delta\diamond\omega$ the
contraction of $\omega$ along $\delta$. In particular, on the level of
$1$-forms, $\diamond$ denotes the natural pairing between $TX$ and
$T^*X$ with values in the $0$-forms $\calO_X$.
\end{ntn}

Contraction of $\Omega^i_X(\log f)$ by a logarithmic
derivation induces an $\calO_X$-morphism to $\Omega^{i-1}_X(\log f)$.
In particular, for an Euler field $E_U$ for $f$ on the open set
$U\subseteq X$ and for $\alpha\in\Omega^i_U(\log f)$,
\begin{eqnarray}\label{eqn-decompose-alpha}
\alpha&=&\alpha\wedge\left(
E_U\diamond(\dff)\right)=
\left(E_U\diamond(\alpha\wedge\dff)-(E_U\diamond
\alpha)\wedge\dff\right)\cdot (-1)^{i}.
\end{eqnarray}

Note that $\{\omega/f\in\frac{1}{f}\Omega_U^i\mid\omega\wedge\de f=0\}=
\{\omega/f\in\Omega_U^i(\log f)\mid\omega\wedge\de f=0\}$ and
define a sheaf $\Omega^\bullet_{X}(\log_E f)$ by
\[
\Omega^i_{U}(\log_E f)=E_U\diamond\left(\Omega^{i+1}_{U}(\log_0 f)\right),
\]
a submodule of $\Omega^i_X(\log f)$, well-defined for the same reasons
that make $\calL_f$ independent of the choice of the coordinate
system: the Jacobian matrix acts (via exterior powers) on $\Omega^i_U$
and (dually) on $\ydx$.

\medskip

\begin{lem}
Suppose that $f\in\calO_X$ has a global Euler field.  The contraction
$E\colon \Omega^i_X(\log f)\to \Omega^{i-1}_X(\log f)$ is injective on
$\Omega^i_{X}(\log_0 f)$, and
\[
\Omega^\bullet_X(\log
f)=\Omega^\bullet_{X}(\log_0 f)\oplus\Omega^\bullet_{X}(\log_E f)\cong 
\Omega^\bullet_{X}(\log_0 f)\oplus\Omega^{\bullet+1}_{X}(\log_0 f).
\] 
\end{lem}
\begin{proof}
The second term in the difference \eqref{eqn-decompose-alpha} of
$\alpha$ is in $\Omega^i_{X}(\log_0 f)$. As $\alpha\wedge\dff$ is in
$\Omega^{i+1}_{X}(\log_0 f)$, the first term of the difference is
in $\Omega^i_{X}(\log_E f)$.  Thus, $\Omega^i_X(\log
f)=\Omega^i_{X}(\log_0 f)+\Omega^i_{X}(\log_E f)$.  If $\alpha\wedge
\dff=0$ and $E\diamond \alpha=0$ for $\alpha\in\Omega^i_X(\log f)$
then \eqref{eqn-decompose-alpha} yields $\alpha=0$. So the sum is
direct.

If $0\not =\alpha\in\Omega^i_{X}(\log_0 f)$, then by
\eqref{eqn-decompose-alpha} $\alpha$ is $(E\diamond \alpha)\wedge\dff$
up to sign, and so surely $E\diamond\alpha$ is nonzero. 
\end{proof}

\begin{rmk}\label{rmk-C}
\begin{asparaenum}
\item \label{rmk-C-pdim} If $f$ permits a global Euler field, then
  $C^i_{f,\x}$ is locally a summand of both $\Omega^i_{X,\x}(\log
  f)[y]$ and $\Omega^{i-1}_{X,\x}(\log f)[y]$, and so (cf.~Remark
  \ref{rmk-Cf}, parts \eqref{rmk-Cf-grading} and \eqref{rmk-Cf-syz})
\[
\pdim_{\calO_{X,\x}[y]}(C^i_{f,\x})\le\min\{\pdim_{\calO_{X,\x}}
\Omega^i_{X,\x}(\log f),\pdim_{\calO_{X,\x}} \Omega^{i-1}_{X,\x}(\log f)\}\le n-2.
\] 
\item The complexes $C^\bullet_f$ and $C^\bullet_{f^k}$ are isomorphic
  up to a shift by $kf^{k-1}$.
\item \label{rmk-C-Koszul}
  If $X=\CC\times X'$ and $f$ does not depend on $x_1$ then
  $C^\bullet_f$ is locally isomorphic to the Koszul cocomplex of $y_1$
  on $\calO_{X}[y_1]$ tensored with the \logarithmic complex
  $C^\bullet_{f_{|_{X'}}}$ to the restriction of $f$ to $X'$.
\item
In a smooth point $\x$ of the reduced hypersurface $\Var(f_\red)$, 
a suitable analytic coordinate change arranges that $f={x_n'}^k$.
In new coordinates, $\de f$ is $k\cdot {x_n'}^{k-1} \de x_n'$ and so
$C^\bullet_{f,\x}$ is (essentially) the Koszul complex in the
indeterminates $y_1',\ldots,y_{n-1}'$ on $\calO_{X,\x}[y]$. In
particular, $C^\bullet_{f,\x}$  is a resolution in $\x$. 
\item If $f$ is Euler-homogeneous the previous item shows that
  $C^\bullet_f$ is exact outside the singular locus of the
  hypersurface $f_\red$. 
\item Over a domain, the wedge product of two $1$-forms is zero if and
  only if they are proportional. Thus, $C^0_f=0$ and $C^1_f=
  \ker(\frac{1}{f}\Omega^1_X[y]\stackrel{\dff\wedge}{\to}\frac{1}{f}
  \Omega^2_X[y])$ is the free module $\calO_X[y]\cdot\frac{f_\red\de
    f}{f^2}$. As $\ydx$ and $\dff$ are generically independent,
  $\ydx\colon C^1_f\to C^2_f$ is injective.  
\schluss
\end{asparaenum}
\end{rmk}

\begin{ntn}
Let $\x\in X$. 
For the rest of this section, put $R=\calO_{X,\x}$ and
$S=\gr_{(0,1)}(\calD_{X,\x})\cong R[y]$,
a standard graded polynomial ring over a Noetherian regular
local ring. 

\end{ntn}


We shall make use of the Acyclicity Lemma of Peskine--Szpiro from
\cite{PeskineSzpiro}:
\begin{thm}
Let $0\to \Lambda^0\to \cdots\to\Lambda^t$ be a complex in the
category of finitely generated modules over the Noetherian ring $A$
such that $\pdim(\Lambda^i)\le i$, and such that the non-vanishing
cohomology modules $H^i(\Lambda^\bullet)$ have depth zero for
$i<t$. If $t\le\depth(A)$ then the complex is a resolution of its
terminal cohomology group.\qed
\end{thm}

\begin{lem}\label{lem-n<5}
The \logarithmic complex is a reflexive resolution of $\calL_f$ if $n\le 3$.
\end{lem}
\begin{proof}
If $n<3$ then the claim follows from Remark \ref{rmk-C}.

Let $f_\nil=\gcd(f_1,\ldots,f_n)$ and
$\phi=\{f_1/f_\nil,\ldots,f_n/f_\nil\}$, recalling that $f_i$ denotes
$\del_i\bullet(f)$.  Then $C^\bullet_f$ and $C^\bullet_\phi$ are equal up to
a twist.
By construction, in each stalk the ideal generated by $\phi$ is either
the unit ideal or has height at least $2$.

If $\phi$ generates the unit ideal at $\x\in X$, then a base change
produces $\phi=(1,0,\ldots,0)$, so that $C^\bullet_\phi$ is
essentially the Koszul cocomplex on $y_2,\ldots,y_n$.  If, on the
other hand, the ideal generated by $\phi$ has height at least $2$ then
$H^i(C^\bullet_f)=0$ for $i\le 2$ by Lemma \ref{lem-phi}.
\end{proof}

In higher dimension, 
strongly Euler-homogeneous Saito-holonomic divisors permit a general
method of arguing that we use to prove Theorem \ref{thm-C-exact}.

\begin{proof}[Proof of Theorem \ref{thm-C-exact}]

We shall proceed by induction on $n$.
Because of Lemma \ref{lem-n<5} we can assume that $n\geq 4$, and that
$f$ is tame, strongly Euler-homogeneous and Saito-holonomic.

Outside $\Var(f)$, $C^\bullet_f$ is exact by Remark \ref{rmk-C}.
Let $\x\in X$ be in a positive-dimensional logarithmic
stratum $\sigma$ of $\Var(f)$. 
By Remarks \ref{rmk-Whitney-diffeo-product} and
\ref{rmk-Ehom}.\eqref{rmk-splitting-strE}, there is an analytic
coordinate change $x\leadsto x'$ near $\x$ transforming a neighborhood
of $\x$ into $\CC^{\dim\sigma}\times \CC^{n-\dim\sigma}$ such that $f$
is constant on the first factor. Let $f'$ be the restriction of $f$ to
the second factor in the new coordinates, and denote by $y\leadsto y'$
the corresponding coordinate change in the cotangent variables. Then
by Remark~\ref{rmk-C}.\eqref{rmk-C-Koszul}, near $\x$ the complex
$C^\bullet_f$ is (relative to the induced
$\gr_{(0,1)}(\calD_X)$-isomorphism from Remark~\ref{rmk-Lf2})
isomorphic to the Koszul complex in $y'_1,\ldots,y'_{\dim\sigma}$,
tensored with the \logarithmic complex of $f'$. Tameness (since
$\Omega^\bullet_X(\log f)=\Omega^\bullet_{X'}(\log f')\otimes_\CC
\Omega^\bullet_{\CC^{\dim(\sigma)}}$), strong Euler-homogeneity (Remark
\ref{rmk-Ehom}), and Saito holonomicity (from the definition) are
inherited from $f$ to $f'$.  Since $f'$ uses fewer than $n$
coordinates, the inductive hypothesis assures that the \logarithmic
complex of $f'$ is a reflexive resolution of the residue ring of
$\calL_{f'}$. Since $C^\bullet_{f'}$ is $y'_i$-torsion-free for
$i\le\dim\sigma$, $C^\bullet_f$ is a reflexive resolution of $\calL_f$
at $\x$.

It follows that $H^{<n}(C^\bullet_f)$ is supported at the
zero-dimensional strata of the logarithmic stratification of $f$.
Then, for all $j$,  the graded components $(C^\bullet_{f,\x})_j$ are complexes of
finitely generated $\calO_{X,\x}$-modules whose $i$-th cohomology is
zero-dimensional for $i<n$.  Since $f$ is tame,
$\pdim_R((C^{n-i}_{f,\x})_j)\le i$ for all $i$ by Remarks
\ref{rmk-C}.\eqref{rmk-C-pdim} and \ref{rmk-Cf}.\eqref{rmk-Cf-grading}
and so $C^\bullet_f$ is, by the
Acyclicity Lemma, a resolution of $\calO_{X,\x}/\calL_f$ everywhere.
\end{proof}

\subsection{Cohen--Macaulayness}

\begin{thm}\label{thm-Lf-CM}
  Let $X$ be a complex manifold, and let $Y$ be a strongly
  Euler-homogeneous and Saito-holonomic divisor cut out by the global
  equation $f\in\calO_X$.  Consider the following conditions:
\begin{enumerate}
\item $f$ is tame: $\pdim_{\calO_X}\Omega^i_X(\log Y)\le i$ for
  all $i\geq 0$.
\item The quotient $\calO_X[y]/\calL_Y$ is Cohen--Macaulay.
\end{enumerate}
Then (1) implies (2) in any case, and if the \logarithmic complex
$C^\bullet_f$ is exact then (2) implies (1).
\end{thm}
\begin{proof}
For strongly Euler-homogeneous divisors, homological
properties of $\calL_f$ depend only on the divisor and not its
defining equation (Remarks \ref{rmk-Der0}, \ref{rmk-Lf2}). 
Tameness, Cohen--Macaulayness and exactness of the \logarithmic 
complex are all local properties. We may hence calculate in the ring
$\calO_{X,\x}[y]$. So, assume that on the open affine set $U\ni\x$ there is a
global section $f$ defining $Y$ that permits a strong Euler field. 

  Since generically $\calL_f$ is a prime of height $n-1$,
  Cohen--Macaulayness is equivalent to $\pdim(\calL_f)=n-1$.  Picture
  $C^\bullet_{f,\x}\cong (\Omega^\bullet_{X,\x}(\log_0 f)[y],\ydx)$ as a row complex.
  Let $F_\bullet^i$ be a minimal graded free
  $\calO_{X,\x}[y]$-resolution (a column oriented downwards) of $C^i_{f,\x}$,
  and let $F^\bullet_\bullet$ be a $y$-graded double complex that results from
  lifting the maps $\ydx\colon C^i_{f,\x}\to C^{i+1}_{f,\x}$. The ring
  $\calO_{X,\x}$ is regular, the complex $C^\bullet_f$ is $y$-graded, and
  $\ydx $ has positive degree. Hence, the total complex to
  $F^\bullet_\bullet$ is minimal. As this total complex resolves
  $C^\bullet_{f,\x}$ under the given hypotheses, by Theorem~\ref{thm-C-exact}
  it minimally resolves $\calO_{X,\x}[y]/\calL_f$, positioned in
  cohomological degree $n$.

  Suppose first that $f$ is tame. Since $\Omega^i_{X,\x}(\log f)\cong
  \Omega^i_{X,\x}(\log_0 f)\oplus \Omega^{i+1}_{X,\x}(\log_0 f)$,
  tameness implies that the length of $F^{i+1}_\bullet$ is at most
  $\pdim_{\calO_{X,\x}}(\Omega^{i+1}_{X,\x}(\log_0
  f))\le\pdim_{\calO_{X,\x}}(\Omega^i_{X,\x}(\log f))\le i$.  Thus, the total complex
  of $F^\bullet_\bullet$ is of length at most $n-1$ and hence
  $\calO_{X,\x}[y]/\calL_f$ is Cohen--Macaulay.

  Conversely, suppose that $f$ is not tame but $C^\bullet_{f,\x}$ is exact.
  Then  the total complex of $F^\bullet_\bullet$ is minimal and free,
  resolves $\calO_{X,\x}/\calL_f$ but has length greater than $n-1$. So $\calL_f$
  cannot be a Cohen--Macaulay ideal at $\x$. 
\end{proof}

\subsection{Integrality}

In general, one has inclusions of $\calO_{X,\x}[y]$-ideals
\begin{eqnarray}\label{eq-3ideals}
\gr_{(0,1)}(\ann_{\calD_{X,\x}} f^s)\supseteq 
\gr_{(0,1)}(\calD_{X,\x}\cdot\Der_{X}(-\log_0 f))\supseteq (\calL_f)_\x.
\end{eqnarray}

In a smooth point of $f_\red$, $\gr_{(0,1)}(\ann_{\calD_X} f^s)$ is
prime of height $n-1$. Hence, wherever $\calL_f$ locally contains a
prime ideal of height $n-1$, all three ideals are identical.  This
makes it very useful to know criteria that assure that $\calL_f$ is
prime, compare Theorem \ref{thm-tame-annfs} below.

For example, if in local coordinates $f\in\calO_U$ can be written as
$f=u\prod_{i=1}^n x_i^{\alpha_i}$ where $u$ is a unit and $\alpha_i>0$
exactly when $1\le i\le k$ then $\Der_U(-\log f)$ is the locally free module
generated locally by the vector fields
$\{(\alpha_j+\frac{x_j\del_j\bullet(u)}{u})x_1\del_1-
(\alpha_1+\frac{x_1\del_1\bullet(u)}{u})x_j\del_j\}_{2\le j\le n}$.
In
particular, $\calL_f$ is prime on $U$ and the previous paragraph
applies: over all normal crossing points of $f$, all three ideals in
\eqref{eq-3ideals} are one and the same
complete intersection prime ideal of height $n-1$. The same argument
has been made more generally wherever $f$ is free and
quasi-homogeneous in
\cite{CalderonNarvaez-theModuleDfs}.


\begin{thm}\label{thm-Lf-prime}
  Let $f\in\calO_X$ be strongly Euler-homogeneous, Saito-holonomic,
  and tame.  Then $\gr_{(0,1)}(\calD_X)/\calL_f$ is an
  integral domain.
\end{thm}
\begin{proof}
  Again, we argue locally.  By the discussion above, $\calL_f$ is
  prime away from singular points of $f_\red$. We argue by induction on
  the strata of the logarithmic stratification $\Sigma_Y$ of
  $f$. Suppose $\sigma$ is an $i$-dimensional stratum and assume it
  has been shown over all strata of dimension larger than $i$ that
  $\calL_f$ is prime.  The local product structure discussed in Remark
  \ref{rmk-Whitney-diffeo-product}.\eqref{rmk-Whitney-splitting}
  implies that the $g$ in that remark inherits from $f$ all hypotheses of
  the present theorem; compare Remark \ref{rmk-Ehom}.
  We can therefore assume that $\sigma$
  is a point $\x$ and calculate in $\calO_{X,\x}[y]$.

  If $\calL_f$ is not prime at the zero-dimensional stratum $\sigma$, it
  has a component primary to a prime ideal that contains the defining
  ideal of $\sigma$.  Such component has dimension $n$ or less
  in the cotangent bundle near $\sigma$. But by Theorem
  \ref{thm-Lf-CM}, $\calL_f$ is a Cohen--Macaulay ideal of generic height
  $n-1$. So all primary components of $\calL_f$ must have the same dimension,
  namely $\dim(\calL_f)=n+1$.  Hence $\calL_f$ has no component over
  $\sigma$, and thus is a prime ideal at $\sigma$.
\end{proof}

\begin{rmk}\label{rmk-Lf-rad}
We have shown that, for Saito-holonomic and strongly Euler-homo\-geneous
$f$, $\calL_f$ is dimension $n+1$ and prime and Cohen--Macaulay over
the tame locus. So for example, with these hypotheses, $\dim
\calL_f=n+1$ for all $f$ with at most one-dimensional non-tame locus
(\emph{e.g.}, when $\dim X\le 5$).

  Saito-holonomicity is necessary for this to be true.  
The divisor $f=xyz(x+y+z)(x+ay+bz)$ in $\CC^5$ is tame, strongly
 Euler-homogeneous, and $C^\bullet_f$ is exact.
 Its logarithmic derivations vanish along the plane $x=y=z=0$
 and thus $\calL_f$ has a component of dimension $7$ and is very much not a
 Cohen--Macaulay ideal of dimension $6$.\schluss
\end{rmk}

\begin{cor}\label{cor-tildeLf}
Let $f\in \calO_X$ be strongly Euler-homogeneous, Saito-holonomic, and
tame.  Then the ideal $\tilde \calL_f$ of $\gr_{(0,1)}(\calD_X[s])$
generated by the symbols of the operators of order one in
$\ann_{\calD_X[s]}(f^s)$ is a Cohen--Macaulay ideal of dimension $n$.
\end{cor}
\begin{proof}
All hypotheses and conclusions are local properties, so we may
calculate in $\calO_{X,\x}[y]$, where $f$ has Euler field $E_\x$.  The
order one operators in $\ann_{\calD_{X,\x}[s]}(f^s)$ are
$\CC[s]\otimes_\CC \Der_{X,\x}(-\log_0 f)\oplus
\calO_{X,\x}[s]\cdot(E_\x-s)$. Hence,
$\tilde\calL_{f,\x}=\calL_{f,\x}+\calO_X[y]\cdot (E_\x\diamond \ydx)$.
As $E_\x\not \in \Der_{X,\x}(-\log_0 f)$, we also have $( E_\x\diamond
\ydx)\not\in \calL_{f,\x}$ as $\calL_{f,\x}$ is graded in $y$.  By
Theorem \ref{thm-Lf-prime}, $E_\x\diamond\ydx$ is regular on
$\calO_{X,\x}[y]/\calL_{f,\x}$.  Hence $\calO_{X,\x}[y]/\tilde
\calL_{f,\x}$ is a quotient of a Cohen--Macaulay ring by a regular
element.
\end{proof}

\begin{rmk}
At every $\x\in X$ where $\Omega^1_X(\log f)$ is free, in the situation of
Corollary \ref{cor-tildeLf},
$\calL_{f,\x}$ and, {\it a fortiori}, $\tilde \calL_{f,\x}$ are local complete
intersections.
In particular, if $f$ is free, satisfies the conditions of Corollary
\ref{cor-tildeLf}, and has a global homogeneity on affine $X$ then
$\calL_f$ and $\tilde \calL_f$ are complete intersections, compare
\cite{Narvaez-Contemp08}. For arrangements, related ideas in a different
context have been worked out in \cite{CDFV-Canadian11}\schluss
\end{rmk}

\begin{rmk}
Tameness might be a red herring in Corollary
\ref{cor-tildeLf}; we know of no case of a strongly Euler-homogeneous and
Saito-holonomic divisor where $\tilde \calL_f$ is not a
Cohen--Macaulay ideal (even including the cases when $\calL_f$ is not a
Cohen--Macaulay ideal).

On the other hand, Saito-holonomicity is necessary: for the $f$ from
Remark \ref{rmk-Lf-rad}, both $\calL_f$ and $\tilde \calL_f$ have
dimension $7$, but are generically (on $X$) of dimensions $6$ and $5$
respectively.\schluss
\end{rmk}


\begin{rmk}
While the hypotheses of Theorem \ref{thm-Lf-CM} apply to $f$ if and
only if they apply to $f^k$ (they have similar
logarithmic vector fields and Euler-homogeneities), it it not clear
whether divisors with the same support are necessarily
simultaneously free or tame.
It would be interesting to understand
the exponents over an arrangement that make it free (or tame); note
that this question can be extended to complex exponents and then
relates to multivariate Bernstein--Sato constructions,
\cite{Budur-BSidealsAndLocalSystems}. For example, the generic
arrangement $xyz(x+y+z)$ is tame but not free, while
$x^ay^bz^c(x+y+z)^d$ is ``a free divisor'' if $a+b+c+d=0$ in the sense
that the relations between the generators of its partial derivatives
form a free module. Similarly, the exponents $(-2,-2,-2,1,1,1,1,1,1)$
used in this sequence on the factors of the non-tame reduced
arrangement in Example \ref{exa-bracelet}  give a tame
``divisor'' in this sense.\schluss
\end{rmk}

\subsection{Blowing up the Jacobian ideal}
The ring $\gr_{(0,1)}(\calD_X)/\calL_f$ is the symmetric algebra of
the Jacobian ideal of $f$: $\calL_f$ is the kernel of the map from
$\Sym(\calO_X^n)\to \Sym(\Jac(f))$ induced by the presentation
$(f_1,\ldots,f_n)$ of $\Jac(f)$.  Symmetric algebras and their
homological properties have been studied intensively for decades,
since they form an approximation to the Rees algebra of an ideal.
Indeed, $\calL_f$ is the linear part (in our grading) of the kernel of
the map that defines the blow-up of $\Jac(f)$.

We record here what we have shown about the symmetric algebra.

\begin{cor}\label{cor-blowup}
Assume that one of the following statements holds:
\begin{itemize}
\item $n\le 3$;
\item $f$ is strongly Euler-homogeneous, tame and Saito-holonomic.
\end{itemize}
Then: 
\begin{enumerate}
\item the regularity of the ideal
$\calL_f$ with respect to the cotangent variables $y_1\ldots,y_n$ is
one: the Jacobian ideal is of linear type;
\item the symmetric algebra
and the Rees algebra of $\Jac(f)$ agree, and are Cohen--Macaulay
domains.
\end{enumerate}
\end{cor}
\begin{proof}
It suffices to consider the local case and we work over
$S=\calO_{X,\x}[y]$.  In both cases, $C^\bullet_f$
is a resolution of $S/\calL_f$ and each $C^i_f$ has a free resolution
whose maps are independent of $y$. The first claim follows from the
fact that the differential in $C^\bullet_f$ is linear in $y$. The second
claim follows from Theorem \ref{thm-Lf-prime} and \ref{thm-Lf-CM}
inasmuch as the symmetric algebra is concerned. If the Rees algebra
were not to agree with the symmetric algebra, it would have to be of
dimension less than $n+1$ under our hypotheses. But that is not
possible, since its projectivization must be dimension $n$.
\end{proof}

\begin{rmk}
\begin{asparaenum}
\item The previous result  generalizes Theorem 5.6 in
\cite{CalderonNarvaez-theModuleDfs} where freeness and
quasi-homogeneity was assumed. 
\item We do not know any instance where $C^\bullet_f$ is not a
  resolution, irrespective of \emph{any} hypotheses.  On the other
  hand, Torrelli pointed out that the ideal
  $\calO_X\cdot f+\Jac(f)$ can be  linear type at $\x$ only if $f$ is
  Euler-homogeneous at $\x$,
  \cite[Rmk.~1.26]{CalderonNarvaez-logCompThm}
\item 
For $f=xy(x+y)$, 
the ideal $\calL_f$ is not  normal.
%
\schluss\end{asparaenum}
\end{rmk}

\subsection{The annihilator of $f^s$}

\begin{lem}\label{lem-know-annfs}
For $\x\in \CC^n$ let $f\in R=\calO_{X,\x}$, and let $\calP$ be a set of
differential operators in $\ann_{\calD_{X,\x}}(f^s)$ such that the
ideal sum in $R[y]$ of $\calL_{f}$ with the ideal generated by
$\gr_{(0,1)}(\calP)$ contains a prime ideal of height $n-1$. Then
$\ann_{\calD_{X,\x}}(f^s)$ is generated by $\Der_{X,\x}(-\log_0 f)$
and $\calP$.
\end{lem}
\begin{proof}
By hypothesis,
$\gr_{(0,1)}(\ann_{\calD_{X,\x}}(f^s))$ contains a prime ideal of
height $n-1$. If $\gr_{(0,1)}(\ann_{\calD_{X,\x}}(f^s))$ were not
equal to this prime ideal, it would be of height $n$ or more. This
contradicts what we know over a generic point of $X$,
but all points have such points in every neighborhood. Hence
$\gr_{(0,1)}(\ann_{\calD_{X,\x}}(f^s))$ is this prime ideal and
generated by the symbols of $\calP$ together with $\calL_f$.

If $Q\in\ann_{\calD_{X,\x}}(f^s)$, then for suitable $a_\delta,b_P\in
\calD_{X,\x}$ (almost all of them zero) we have
\[
\gr_{(0,1)}(Q)=\sum_{\delta\in\Der_{X,\x}(-\log_0
  f)}\gr_{(0,1)}(a_\delta\cdot\delta) +\sum_{P\in \calP}\gr_{(0,1)}(b_P\cdot
P).
\]
For graded (in $y$) ideals in $\calO_{X,\x}$, such as the ideal
generated by $\calL_{f,\x}$ and the symbols of $\calP$, 
this can be arranged in such a way that the
orders of any $a_\delta\cdot\delta$ and $b_P\cdot P$ are equal to 
the order of $Q$. Then $Q$ can be reduced modulo
$\calD_{X,\x}(\calP,\Der_{X,\x}(-\log_0 f))$ to an operator in
$\ann_{\calD_{X,\x}}(f^s)$ of lower order. By induction on this
  order, the lemma follows.
\end{proof}

\begin{thm}\label{thm-tame-annfs}
  If $f\in\calO_X$ is tame, strongly Euler-homogeneous and Saito
  holonomic then $\ann_{\calD_{X}[s]}(f^s)$ is generated at $\x\in X$
  by $\Der_X(\log_0(f))$ and any Euler-homogeneity $E_\x-s$ where
  $E_\x\bullet(f)=f$. If $X$ is the analytic space to a smooth
  $\CC$-scheme, this holds also in the algebraic category.
\end{thm}
\begin{proof}
Locally, this follows from Theorem \ref{thm-Lf-prime} and Lemma
\ref{lem-know-annfs}. So the containment of sheaves
$\calD_X\cdot\Der_{X}(-\log_0 f)\subseteq \ann_{\calD_X}(f^s)$ is an
isomorphism. Now note that algebraic $f$ has algebraic derivatives;
all syzygies between the derivatives of $f$ will then be algebraic.
%
\end{proof}

\begin{exa}
Let $R=\CC[x_{1,1},\ldots,x_{m,n}]$ be a polynomial ring and let $M$
be the matrix with $m_{i,j}=x_{i,j}$. It is well-known that for
$n=m+1$ the product $f$ of the $m+1$ maximal minors of $M$ is a
(linear) free divisor, see \emph{e.g.}
\cite{GrangerMondNietoSchulze-Fourier09}. When $m=3=n-1$,
J.~Mart\'in-Morales informs us that (on the behest of
L.~Narvaez-Macarro) V.~Levandovskyy and D.~Andres have used the
computer algebra system \emph{Singular} \cite{Singular} to determine
that the annihilator of $f^s$ is not cut out by operators of order
one. Experimental evidence lets us believe that $f$ is strongly
Euler-homogeneous everywhere. Theorem \ref{thm-tame-annfs} indicates
that $f$ should not be Saito-holonomic. And indeed, the locus where
the logarithmic vector fields have rank less than 10 is
10-dimensional: it contains the variety of rank-deficient matrices.
\schluss
\end{exa}

\section{Milnor fiber and Jacobian module}\label{sec-Milnor}

In this section, $n\geq 2$ and $X$ is the algebraic variety $\CC^n$.
If $f\in\Gamma(X,\calO_X)=:R_n$ denotes a
homogeneous polynomial of degree $d$ in $n$ variables, then the Milnor
fiber $M_{f,0}$ of $f$ at the origin can be identified with the hypersurface
$(f=1)$. 

If $f$ has an isolated singularity at the origin $0\in X$, Milnor 
established that the cohomology of $M_{f,0}$ is
encoded in the residue ring of the Jacobian ideal:
the Milnor fiber is a bouquet of $\mu=\dim_\CC(\calO_{X,0}/\Jac(f))$ many
$(n-1)$-spheres.

Theorem \ref{thm-jac-milnor} below can be seen as a partial extension
to a class of more general homogeneous singularities. In these cases,
the Jacobian ring is not necessarily Artinian, hence only parts of it
can be responsible for cohomology classes on $M_f$. A relevant part
of $R_n/\Jac(f)$ is selected by a local cohomology functor and we
connect it to Hodge-theoretic data on $M_f$ via logarithmic vector
fields.

\begin{rmk}\label{rmk-localcoh}
Let $I$ be an ideal in a commutative ring $R$.  We will make use of
Grothendieck's {\em local cohomology functors} $H^\bullet_I(-)$, the
right derived functors of the left-exact $I$-torsion functor
$H^0_I(-):=\bigcup_t\Hom_R(R/I^t,-)$. For details we refer to
\cite{24h}.\schluss
\end{rmk}

\begin{ntn}
We use the term \emph{simple normal crossing} to label a point
in a divisor where locally a coordinate system exists in which the
divisor takes the form $x_1^{a_1}\cdots x_n^{a_n}$, $a_i\in\NN$, and
where $(x_i=0)$, $(x_j=0)$ define globally distinct components of the divisor
whenever $i\not =j$.
The points of the variety $(g=0)$  where $g$ does not have
simple normal crossings is called the \emph{NSNC-locus of $g$}.
\end{ntn}

\begin{thm}\label{thm-jac-milnor}
Let $f\in R_n$ be homogeneous of degree $d$, reduced, and suppose $n\geq
2$. Assume that $\Proj(R_n/f)$ has isolated singularities.
Then, with
$1\le k\le d$ and 
$\lambda=\exp(2\pi \sqrt{-1}k/d)$,
\begin{eqnarray*}
\dim_\CC [H^0_\frakm(R_n/\Jac(f))]_{d-n+k}&\le& 
\dim_\CC \gr^{{\rm Hodge}}_{n-2}(H^{n-1}(M_{f,0},\CC)_\lambda)
\end{eqnarray*}
where the right hand side indicates the $\lambda$-eigenspace of the
associated graded object to the Hodge filtration on $H^{n-1}(M_{f,0},\CC)$.
\end{thm}

\begin{proof}
We first give an outline of the proof of Theorem
\ref{thm-jac-milnor} and then fill in the details in five steps. 

Let $S_n=\CC[x_0,x_1,\ldots,x_n]$ be the coordinate ring of projective
$n$-space. We usually denote $x_0$ by $z$ and identify
$\PP^n\smallsetminus \Var(z)$ with $X=\CC^n$. Then define a new
polynomial $F\in S_n$ by
$F(z,x_1,\ldots,x_n)=(f(x_1,\ldots,x_n)-z^d)z$.
The corresponding homogeneous maximal ideals are
denoted $\frakm=(x_1,\ldots,x_n)R_n$ and
$\frakn=(x_0,\ldots,x_n)S_n$. The projective scheme defined by $F$ is
denoted $Y\subseteq\PP^n$.

We show first that the 
singular
locus of $\Proj(S_n/(f,z))$ is the same as the
NSNC-locus of $\Proj(S_n/F)$ 
and so the hypotheses imply that the
NSNC-locus of $\Proj(S_n/F)$ is zero-dimensional.

Next, choose an embedded resolution $\pi\colon (\PP,Y')\to (\PP^n,Y)$
of singularities that resolves only the NSNC-locus of
$\Proj(S_n/F)$. As $f$ is homogeneous, the resolution process takes
place entirely inside the (preimage of the) locus defined by
$\Var(z)$. Since the Milnor fiber $M_{f,0}$ is $\Var(f-1)\subseteq X$,
it agrees with $\Var(f-z^d)\smallsetminus\Var(z)$ inside
$\PP^n_\CC$. In the resolved model, the strict transform of $f-z^d$ is
smooth, and so the cohomology of $M_{f,0}$ is determined by the
cohomology of logarithmic vector fields on the strict transform of
$\Var(f-z^d)$ along (the preimage of the reduced divisor defined by)
$\Var(z)$. Long exact sequences can be used to translate the issue to
cohomology of logarithmic vector fields on $\PP$ along (the reduced
divisor defined by) $\Var(F)$.  If the NSNC-locus of $\Proj(S_n/F)$ is
zero-dimensional, it can be shown that the appropriate cohomology of
logarithmic vector fields along $\Var(\pi^*(F))$ and along its reduced
divisor are closely related: the Leray spectral sequence allows to
connect both to cohomology of logarithmic vector fields on $\PP^n$
along $\Var(F)$.  We then use local cohomology methods to relate such
cohomology to torsion in the Jacobian ring of $F$ and of $f$.

\stp{0}{The NSNC-locus of  $\Proj(S_n/F)$.}

We have $\Jac(F)=(z\Jac(f),(d+1)z^d-f)S_n$ so that the singular locus
of $\Proj(S_n/F)$ is defined by $(z=f=0)$. In particular, the singular
(and hence the NSNC-) locus of $F$ are inside $(z=0)$.

Without loss of generality we take the chart $(x_n\not =0)$ of
$\Proj(S_n)$ and show that in this chart all points outside
$(z,\Jac(f))S_n$ are simple normal crossing points of $F$. In terms of
coordinates on this chart, write $y_i=x_i/x_n$ and $t=z/x_n$, and write
$g(y_1,\ldots,y_{n-1})$ for $f(x_1,\ldots,x_n)/x_n^d$. 

Imagine a point $p$
outside $\Var(z,\Jac(f))$ that fails to have simple normal
crossings. Since $p$ must be a singular point of $F$, we will have
$z=f=0$ at $p$ 
and so $p$ is in $\Var(z)$ but (therefore) not in $\Var(\Jac(f))$.
So, some derivative $f_\ell$ does not vanish at $p$ and so $f-z^d$ is
smooth in $p$. 

In order for $p$ to be in the NSNC-locus of $F$ it is necessary that
the gradients of $t$ and of $g-t^d$ be linearly dependent at $p$,
which implies that the gradient of $g$ must be zero in $p$. As
$g(y_1,\ldots,y_{n-1},1)=f(x_1,\ldots,x_n)/x_n^d$ the chain rule implies
that $\frac{\del (f(x_1,\ldots,x_n)/x_n^d)}{\del x_i}=\sum_{j=1}^{n-1} \frac{\del g}{\del y_j}
\cdot\frac{\del y_j}{\del x_i}$. If $\ell\not =n$ this states that
$\frac{\del f}{\del x_\ell}(y_1,\ldots,y_{n-1},1) = \frac{\del g}{\del
  y_\ell}$ while for $\ell=n$ it states that
$(-d)x_n^{-d}f(x_1,\ldots,x_n)+x_n^{-d+1}f_n(x_1,\ldots,x_n)=-\sum_{j=1}^{n-1}
y_j\cdot g_j(y_1,\ldots,y_{n-1})$.

As the gradient of $g$ vanishes at $p$, so do all derivatives of
$g$. By the chain rule above, each $f_j$, $1\le j\le n-1$, is zero in
$p$. By the chain rule for the $x_n$-derivative above, since $f$
vanishes at $p$ but $x_n$ does not, we conclude that $f_n$ is zero at
$p$ as well. In other words, $p$ is in $\Var(\Jac(f),z)$. By contradiction, the
NSNC-locus of $F$ is contained in $\Var(z,\Jac(f))$.

Conversely, take $p\in\Var(z,\Jac(f))\subseteq\Var(z,z^d-f)$ on the chart
$x_n\not=0$. Then the gradient of $g$ is zero at $p$ and so $p$ is an
NSNC-point of $F$ as $(g=t^d)$ and $(t=0)$ do not meet normally at $p$.

\stp{1}{From $R_n/\Jac(f)$ to $\Omega^{n-1}_{\PP^n}(\log F)$.} 

As
$R_n$-module, $S_n/\Jac(F)$ is generated by the cosets of
$z^0,\ldots,z^{d-1}$, and these are 
$R_n$-independent:
\[
S_n/\Jac(F)\cong \bigoplus_{i=0}^{d-1} R_n\cdot(z^i\modulo \Jac(F)).
\]
Observe that the $R_n$-annihilator of $(z^i\modulo\Jac(F))$ 
is $\Jac(f)$ if $1\le
i\le d-1$, and that the $R_n$-annihilator of
$(1\modulo\Jac(F))$ is $f\cdot \Jac(f)$.  Since $z$ is nilpotent on
$S_n/\Jac(F)$, the concepts of $\frakm$-torsion and $\frakn$-torsion
agree on this module. Then, as graded $R_n$-modules,
\[
H:=H^0_\frakn(S_n/\Jac(F))=\bigoplus_{i=1}^{d-1}H^0_\frakm(R_n/\Jac(f))(-i)\oplus
H^0_\frakm(R_n/\Jac(f))(-d)
\]
since 
$H^0_\frakm(R_n/f\cdot\Jac(f))=H^0_\frakm(f\cdot R_n/f\cdot\Jac(f))
=H^0_\frakm(R_n/\Jac(f))(-d)
$.

\medskip

Now let $Z\subseteq \bigoplus_{i=0}^nS_n\cdot e_i$ be the syzygy module
on the partial derivatives $F_z=F_0,F_1,\ldots,F_n$ where $F_i=\del
F/\del x_i$. It inherits a
natural grading via $\deg(a_ie_i)=\deg(a_i)+d$, since
$\deg(F_i)=d$.  The start of the minimal $S_n$-graded resolution of
$S_n/\Jac(F)$ is hence
\[
0\to Z\to S_n(-d)^{n+1}\to S_n\to S_n/\Jac(F)\to 0.
\] 
It follows from the long exact sequence of local cohomology that, as
graded modules, $ H=H^0_\frakn(S_n/\Jac(F))\cong H^1_\frakn(\Jac(F))
\cong H^2_\frakn(Z).  $

The module $\Der_{S_n}(-\log_0 F)$ of derivations
annihilating $F$  inherits a natural
grading via $\deg(\sum_{i=0}^na_i\del_i)= \deg(a_i)-1$.  
As graded
modules, $\Der_{S_n}(-\log_0 F)\cong Z(d+1)$. Thus,
\[
H^2_\frakn(Z)\cong H^2_\frakn(\Der_{S_n}(-\log_0 F))(-d-1).
\]

Since $F$ is homogeneous, the
Euler derivation $E=\sum_{i=0}^n x_i\del_i$ is in
$\Der_{S_n}(-\log F)$, and there is a splitting
$S_n\cdot E\oplus \Der_{S_n}(-\log_0 F)\cong
\Der_{S_n}(-\log F)$.
Since
$S_n\cdot E$ is free and $n\geq 2$, the splitting gives
\[
H^2_\frakn(\Der_{S_n}(-\log_0 F))(-d-1)\cong
H^2_\frakn(\Der_{S_n}(-\log F))(-d-1).
\]

The differential forms $\Omega_{S_n}^\bullet$ on $\CC^{n+1}$ form a
graded algebra with differential $\de\colon
\Omega^0_{S_n}=S_n\to\Omega^1_{S_n}$ homogeneous of degree zero.  The
modules $\Omega_{S_n}^\bullet(\log F)$ inherit a natural grading from
being a submodule of $\frac{1}{F}\cdot\Omega_{S_n}^\bullet\cong
\Omega^\bullet_{S_n}(d+1)$, and as graded modules $\Der_{S_n}(-\log
F)\cong \Omega^n_{S_n}(\log F)(n-d)$, cf.~Remark
\ref{rmk-Cf}.\eqref{rmk-Cf-fields-diffs}.

The module $\Omega^{n-1}_{S_n}(\log_E F)\cong \Omega^n_{S_n}(\log_0
F)$ is the contraction of $\Omega_{S_n}^n(\log F)$ with $E$; this
surjection is dual to the inclusion $\Der_{S_n}(-\log_0 F)\into
\Der_{S_n}(-\log F)$ under the above identification. Hence there is an
induced graded isomorphism
\[
\Der_{S_n}(-\log_0 F)\cong \Omega_{S_n}^{n-1}(\log_E F)(n-d)
\] 
sending $\delta\mapsto
\delta\diamond(E\diamond(\de x/F))$, compare \cite{DenhamSchulze}.

Pasting together all previous identifications, 
\begin{eqnarray*}
H^2_\frakn(\Omega^{n-1}_{S_n}(\log_E F))
&\cong&H^2_\frakn(\Der_{S_n}(-\log_0 F))(d-n)\\
&\cong&H^2_\frakn(Z)(2d -n+1)\\
&\cong&H(2d-n+1)\\
&\cong&\bigoplus_{i=1}^{d}H^0_\frakm(R_n/\Jac(f))(d-n+i).
\end{eqnarray*}

\medskip

Let $\tilde \Omega^{n-1}_{S_n}(\log_E F)$ be the sheaf on $\PP^n$
induced by $\Omega^{n-1}_{S_n}(\log_E F)$. By \cite[Proposition
  2.11]{DenhamSchulze} (valid for general homogeneous divisors),
$\tilde \Omega^{n-1}_{S_n}(\log_E F)$ is the sheaf
$\Omega^{n-1}_{\PP^n}(\log F)$ of $(n-1)$-forms on $\PP^n$ that are
logarithmic in the sense of K.~Saito along $Y=\Proj(S/F)$ on all charts.  The
Grothendieck--Serre correspondence
\cite[Thm.~20.4.4]{BrodmannSharp-LC} yields then a
graded $S_n$-module isomorphism
\[
H^2_\frakn(\Omega^{n-1}_{S_n}(\log_E F)) \cong \bigoplus_{t\in\ZZ}
H^1(\PP^n,\Omega^{n-1}_{\PP^n}(\log F)(t)).
\]

In particular, 
\begin{eqnarray*}
\bigoplus_{i=1}^d\left[H^0_\frakm(R/\Jac(f))\right]_{d-n+i}&=&
\left[\bigoplus_{i=1}^d H^0_\frakm(R/\Jac(f))(d-n+i)\right]_0\\
&=&
\left[H^2_\frakn(\Omega^{n-1}_{S_n}\log_E F)\right]_0=
H^1(\PP^n,\tilde\Omega^{n-1}_{S_n}(\log_E F))\\
&=&
H^1(\PP^n,\Omega_{\PP^n}^{n-1}(\log F)).
\end{eqnarray*}

\stp{2}{The Leray spectral sequence.}

Let $\pi\colon (\PP,Y')\to(\PP^{n},Y)$ be an embedded resolution of
singularities, with $\PP$ smooth and $F':=\pi^*(F)$ having simple normal
crossings.  Since the singularities of $F$ are inside $\Var(z)$, one
can arrange that $\PP$ and $\PP^n$ agree where $z\neq 0$.  Indeed,
functoriality of the resolution process implies that one only needs to
resolve the NSNC-locus of $\Proj(S_n/F)$.

Note that while $F$ is reduced, this is not so for $F'$. 
Let
$F'_\red$ define the reduced divisor. 
Deligne's logarithmic sheaves $\Omega^i_\PP(Y')$ consider
reduced underlying schemes and are our
sheaves $\Omega^i_\PP(\log(F'_\red))$, usually properly contained in
$\Omega^i_\PP(\log(F'))$.

The cokernel $Q$ of the inclusion $\pi_*(\Omega^{n-1}_\PP(
\log F'_\red))\subseteq \pi_*(\Omega^{n-1}_\PP(\log F'))$ is supported in the
NSNC-locus of $F$. If this locus is of dimension zero,
$H^{>0}(\PP^n,Q)=0$ and so there is a natural surjection
\begin{gather}\label{eqn-diffs}
H^1(\PP^n,\pi_*(\Omega_{\PP}^{n-1}(\log F'_\red)))\onto
H^1(\PP^n,\pi_*(\Omega_{\PP}^{n-1}(\log F')))
\end{gather}
(and the higher cohomology groups agree).

On the other hand, 
the Leray spectral sequence
induces a natural embedding
\[
H^1(\PP^n,\pi_*(\Omega_{\PP}^{n-1}(\log F'))\into
H^1(\PP,\Omega_\PP^{n-1}(\log F')).
\]
By Lemma \ref{lem-res-omega}, $\pi_*(\Omega^{n-1}_\PP(\log
F'))=\Omega^{n-1}_{\PP^n}(\log F)$ and so we have a natural diagram
\begin{eqnarray*}
H^1(\PP^n,\Omega_{\PP^n}^{n-1}(\log F))=
H^1(\PP^n,\pi_*(\Omega_{\PP}^{n-1}(\log F')))
\into 
&H^1(\PP,\Omega_\PP^{n-1}(\log F'))\\
&{\mathrel{\rotatebox{90}{$\twoheadrightarrow$}}}\\
&H^1(\PP,\Omega_\PP^{n-1}(\log F'_\red))
\end{eqnarray*}
as long as $H^1(\PP^n,Q)=0$.  In particular, the dimension of
$H^1(\PP^n,\Omega_{\PP^n}^{n-1}(\log F))$ is bounded above by the
dimension of
$H^1(\PP,\Omega_\PP^{n-1}(\log F'_\red))=H^1(\PP,\Omega_\PP^{n-1}(Y'))$.

\stp{3}{Matching with Deligne's results.}

Consider for the moment a smooth projective variety $W$ with a reduced
simple normal crossing divisor $V=\bigcup V_i$ and a distinguished
component $V_0$. 
Then, by \cite[Properties~2.3]{EsnaultViehweg} there is a short exact
sequence of sheaves
\[
0\to \Omega^i_W(\log(V-V_0))\to \Omega^i_W(\log (V))\to 
\Omega^{i-1}_{V_0}(\log(V- V_0)|_{V_0})\to 0.
\]
With $i=n-1$, it induces a long exact sequence
\begin{gather}
\label{seq-4term}
\ldots\to H^1(W,\Omega^{n-1}_W(\log(V- V_0)))\to 
H^1(W,\Omega^{n-1}_W(\log (V)))\to\hspace{3cm}\\\nonumber
\hspace{3cm}\to H^1(V_0,\Omega^{n-2}_{V_0}(\log(V- V_0)|_{V_0}))\to
H^2(W,\Omega^{n-1}_W(\log(V- V_0)))\to\ldots
\end{gather}

Deligne's theory \cite{Deligne-HodgeII,Deligne-HodgeIII} implies that  
\begin{eqnarray*}
H^i(W,\Omega^j_W(\log
(V-V_0)))&\cong&\gr^{{\rm Hodge}}_jH^{i+j}_{\dR}(W\minus (V-V_0));\\ 
H^i(W,\Omega^j_W(\log
V))&\cong&\gr^{{\rm Hodge}}_jH^{i+j}_{\dR}(W\minus V);\\ 
H^i(V_0,\Omega^j_{V_0}(\log
(V-V_0)|_{V_0}))&\cong&\gr^{{\rm Hodge}}_jH^{i+j}_{\dR}(V_0\minus (V-V_0)). 
\end{eqnarray*}
Here, $V-V_0$ is the difference of divisors, $W\minus V$ the
set-theoretic difference.

The situation we are interested in is when $W=\PP$, $V=Y'$, and
$V_0$ is the strict transform of $\Div(f-z^d)$. 
In that case, $W\minus(V-V_0)$ is affine space
$\CC^n=\PP^n\minus\Var(z)$ while
 $V_0\minus(V-V_0)$ is the smoothly compactified Milnor fiber
minus the part at infinity, hence exactly $M_{f,0}$.
In the 4-term exact sequence \eqref{seq-4term}, the left- and
right-most terms are zero as $\CC^n$ is contractible; we obtain a
monomorphism $H^1(\PP,\Omega^{n-1}_\PP(\log F'_\red)) \into
\gr^{{\rm Hodge}}_{n-2}H^{n-1}(M_{f,0},\CC)$.

Collecting the results from each step, one obtains
\begin{eqnarray*}
\bigoplus_{i=1}^d\left[H^0_\frakm(R_n/\Jac(f))\right]_{d-n+i} =
H^1(\PP^n,\Omega_{\PP^n}^{n-1}(\log F))
\into 
&H^1(\PP,\Omega_\PP^{n-1}(\log F'))\\
&{\mathrel{\rotatebox{90}{$\twoheadrightarrow$}}}\\
&H^1(\PP,\Omega_\PP^{n-1}(\log F'_\red))\\
&{\mathrel{\rotatebox{270}{$\into$}}}\\
&\gr^{{\rm Hodge}}_{n-2}(H^{n-1}(M_{f,0},\CC)).
\end{eqnarray*}

\stp{4}{Monodromy.}

The monodromy isomorphism on $H^i(M_{f,0})$ is induced by the
geometric monodromy $x_i\to \omega_d x_i$ where $\omega_d$ is a
primitive $d$-th root of unity. One extends it to a graded
automorphism of $S_n$ that fixes $\Var(F)$. Thus, the grading and the
geometric monodromy extend to the pair $(\PP^n,Y)$ and functoriality
of resolution of singularities \cite{jarek} guarantees an extension to
$\pi$. In particular, all morphisms in the above display are graded
and monodromy-equivariant (apart from the combined effect of
identification of logarithmic derivations with logarithmic
differentials and the residue map, corresponding to a twist by
$(\omega_d)^n$). Since the graded components of the module
$H^0_\frakn(S_n/\Jac(F))$ are $\omega_d$-eigenspaces, the claim follows.
\end{proof}

\begin{rmk}
\begin{asparaenum}
\item If $n=3$, the condition on the singular locus is automatic.
\item The current form of Theorem \ref{thm-jac-milnor} evolved from a
  previous one through
  helpful comments and criticisms by M.~Saito. He has produced a similar proof
  that is based on cyclic covers rather than the closure of the Milnor
  fiber in $\PP^n$, see the appendix of \cite{Saito-HilbertBernstein2015}.
\item If the natural map \eqref{eqn-diffs} is an isomorphism, one
  obtains a natural map $[H^0_\frakm(R_n/\Jac(f))]_{d-n+i}\into 
\gr^{{\rm Hodge}}_{n-2}(H^{n-1}(M_f,\CC)_\lambda)$. There seems to be no natural
lift to $H^{n-1}(M_f,\CC)_\lambda$, as was pointed out by M.~Saito.
\item Nonzero elements in $H^{n-1}(M_f,\CC)$ can be related to roots
  of $b_f(s)$ under certain circumstances, see for example
  \cite{Saito-Bull94,Saito-bfu,W-Bernstein}. M.~Saito has produced
  examples that show that there is no obvious relation between
  $b_f(s)$ and $H^0_\frakn(S_n/\Jac(F))$, even if \eqref{eqn-diffs} is
  an isomorphism, see \cite{Saito-HilbertBernstein2015}.\schluss
\end{asparaenum}
\end{rmk}

\section{Hyperplane arrangements}

In this section we consider hyperplane arrangements 
$\calA\subseteq \CC^n=:X$ defined by 
\[
f_\calA=\prod_1^d L_i\in R_n:=\CC[x], 
\]
where the $L_i$ are (not necessarily homogeneous and not necessarily
distinct) polynomials of degree one. Arrangements form an interesting
class of strongly Euler-homogeneous and Saito-holonomic divisors.  The
canonical Whitney stratification, the logarithmic stratification, and
the stratification by multiplicity all agree for arrangements.

Let $D_n$ be the $n$-th Weyl algebra
$R_n\langle\del\rangle$ and set as before
\[
E=\sum_{i=1}^n x_i\del_i\in D_n=R_n\langle\del\rangle.
\] 

\subsection{The differential annihilator}

Around the time \cite{Terao-JAlg02} was published, Terao conjectured
that $\ann_{D_n}(1/f_\calA)$ be generated by operators of order one
for every reduced hyperplane arrangement.  In
\cite{CastroGagoHartilloUcha-Revista07} it is shown that for locally
weakly quasi-homogeneous free polynomials $f$ (which includes free
arrangements), $\ann_{D_n}(1/f^k)$ is generated by operators of order
one for
$k\gg 0$, compare also \cite[Rem.~1.7.4]{Narvaez-Contemp08}.  Terao's
conjecture was affirmed for generic arrangements by Torrelli, who also
raised the corresponding question about the generic annihilator
$\ann_{D_n}(f_\calA^s)$ \cite{Torrelli-BullMathSoc04}. Related results
are contained in \cite{Holm-CiA04}.

We prove here that Terao's
original conjecture is correct. Moreover, we confirm the generic 
annihilator conjecture in a large class of examples, but disprove
it in general.
\begin{dfn}
  A \emph{central} arrangement is an arrangement defined by a polynomial
  $f_\calA$ that is homogeneous in the standard sense.
\end{dfn}

The following is a direct consequence of Theorem \ref{thm-tame-annfs}:
\begin{thm}\label{thm-generic-Terao}
  If the arrangement $\calA\subseteq \CC^n$ is tame then
  $\ann_{D_n}(f_\calA^s)$ is generated by operators of order one.  \qed
\end{thm}

This result implies immediately the validity of Terao's conjecture for
tame arrangements. Indeed, we can assume that $f_\calA$ is central of
degree $d$, with a global Euler-homogeneity $E$. Then
$\ann_{D_n[s]}(f_\calA^s)$ is (globally) generated by operators of
order one, namely by $E-ds$ and $\Der_{\CC^n}(-\log_0 \calA)$. The only
integral root of the Bernstein--Sato polynomial of an arrangement is
$-1$ by \cite{W-Bernstein}. Consider the evaluation morphism
$\ann_{D_n[s]}(f^s_\calA)\to \ann_{D_n}(1/f_\calA)$ that sends $s$ to
$-1$. By \cite[Prop.~6.2]{Kashiwara-bfu}, this morphism is surjective,
hence the target is generated by order one operators. We show next that
tameness is not required: Terao's conjecture holds in fact in full generality.

\begin{thm}\label{thm-terao}
The differential annihilator $\ann_D(1/f_\calA)$ is generated by operators
of order one.
\end{thm}
\begin{proof}

If the rank $r$ of the arrangement is less than $n$ then a suitable change
of coordinates brings $f_\calA$ into the ring $\CC[x_1,\ldots,x_r]$. The
truth of the statement of the theorem then only depends on the
corresponding arrangement in $\CC^r$. 
Without loss of generality we hence may (and will) assume that $r=n$.

Before getting to the main part of the proof we need to set up some
notation.  Denote for a subset $I$ of $\{1,\ldots,d\}$ by $L_I$ the
factor $\prod_{i\in I} L_i$ of $f_\calA$.  Let
\[
\calL^\bullet\colon 0\to \underbrace{R_n}_{\text{degree } 0}\to
\underbrace{\bigoplus R_n[1/L_i]}_{\text{degree }1}\to
\cdots\to\underbrace{\bigoplus_{|I|=d-1}R_n[1/L_I]}_{\text{degree
  }d-1}\to \underbrace{R_n[1/f_\calA]}_{\text{degree }d}\to 0
\]
be the \v Cech complex attached to the family
$\{L_1,\ldots,L_d\}$. Since the rank of $\calA$ is $n$, the ideal
generated by $L_1,\ldots,L_d$ is the maximal graded ideal
$\frakm=R_n\cdot x$ where $x=(x_1,\ldots,x_n)$. By standard facts of
local cohomology (see \cite{24h}), the complex $\calL^\bullet$ has
exactly one cohomology group, positioned in cohomological degree $n$.

The \v Cech complex is a complex of $D_n$-modules since
its constituents are localizations of the tautological $D_n$-module
$R_n$. As such, its unique cohomology group is the direct image of the
structure sheaf of the origin under the map that embeds the origin in
$\CC^n$. In particular, $H^n(\calL^\bullet)$ is isomorphic to
$D_n/D_n\cdot x$. 

For $I\subseteq \{1,\ldots,n\}$, let $\eta_I$ be the element in
$\calL^{|I|}$ whose component in $R_n[1/L_I]$ is $1/L_I$, and which is
zero in all other components.  The coboundary map of $\calL^\bullet$
is just the signed sum of all inclusions $R_n[1/L_I]\into
R_n[1/L_{I\cup\{j\}}]$. It thus sends $\eta_I$ to $\sum
(-1)^{\sgn(I,j)}L_j\cdot\eta_{I\cup \{j\}}$ where the sum runs over
elements $j$ in $\{1,\ldots,d\}\smallsetminus I$ and where $\sgn(I,j)$
is $-1$ raised to the number of elements of $I$ that are greater than
$j$.

Now let $\calD^\bullet$ be the Koszul cocomplex associated to
right-multiplication by the elements $L_1,\ldots,L_d$ on $D_n$. The
module $\calD^k$ is the direct sum of copies $D_n\cdot e_I$ of $D_n$ where
$I$ runs over the subsets of $\{1,\ldots,d\}$ of size $k$ and $e_I$ is
simply a symbol. With the right choices, the coboundary map in
$\calD^\bullet$ sends $1\cdot e_I$ in $\calD^{k}$ to $\sum
(-1)^{\sgn(I,j)}L_j\cdot e_{I\cup\{j\}}$ in $\calD^{k+1}$.

There is a morphism of complexes $\phi\colon \calD^\bullet\to
\calL^\bullet$ that sends $e_I$ to $\eta_I$. Since $-1$ is the only
integral root of the Bernstein--Sato polynomial of an arrangement,
$\eta_I$ generates $R_n[1/L_I]$ over $D_n$ and so $\phi$ is
surjective. For each index set $I$ denote by $K_I$ the kernel of the map
$\phi_I\colon D_n\cdot e_I\to D_n\bullet \eta_I=R_n[1/L_I]$. Then there are
induced maps $K_I\to K_{I\cup\{j\}}$ via right-multiplication by $L_j$
for each 
$j\in\{1,\ldots,d\}\smallsetminus I$ and thus an induced short exact
sequence
\[
0\to \calK^\bullet\to \calD^\bullet\to \calL^\bullet\to 0
\]
of complexes, where $K_I$ becomes the $I$-th constituent of
$\calK^t=\bigoplus_{|I|=t}K_I$.

We now start the actual proof of the theorem.
Since the rank of our arrangement is $n$, for
$d=n$ we are looking at a Boolean arrangement and in that case the
theorem is easy to check. Assume now that $d>n$. The long exact
sequence of cohomology attached to the above short exact sequence of
complexes ends with
\begin{gather}\label{eqn-les}
H^{d-1}(\calD^\bullet)\stackrel{\alpha}{\to} H^{d-1}(\calL^\bullet)\to
H^d(\calK^\bullet)\stackrel{\beta}{\to}
\underbrace{H^d(\calD^\bullet)}_{=D_n/D_n\cdot x}\to H^d(\calL^\bullet)=0,
\end{gather}
compare the second paragraph of the proof.

Our first contention regarding this sequence is that $\alpha$ is
surjective. When $d-1>n$ then this is automatic as then the target
module vanishes. In the case $d-1=n$ denote $L_{\hat r}$ the
polynomial $f_\calA/L_r$ and observe that the kernel of $\calL^{d-1}\to
\calL^d$ consists of the cochains $(g_1/L_{\hat
  1}^k,\ldots,g_d/L_{\hat d}^k)\in\calL^{d-1}$ for which $\sum
(-1)^jg_jL^k_j=0$. Since $d-1=n$ equals the rank of the arrangement,
there is an (up to scaling) unique $\CC$-linear syzygy $\sum (-1)^rc_r
L_r=0$, $c_r\in\CC$, between $L_1,\ldots,L_d$. The corresponding
element $c_L:=\oplus c_r/L_{\hat r}\in \calL^{d-1}$ is a cochain and
at least one of the polynomials $L_{\hat r}$ with $c_r\not=0$ is the
product of $n$ linearly independent linear forms. We briefly consider
the \v Cech complex on the forms
$L_1,\ldots,L_{r-1},L_{r+1},\ldots,L_d$ generating the ideal
$\frakm$. The $n$-th cohomology group of this complex is
$H^n_{\frakm}(R_n)=D_n/D_n\cdot x$ and generated over $D_n$ by $c_r/L_{\hat
  r}$. It follows that $c_r/L_{\hat r}$ is not a sum of fractions
where each denominator is a proper factor of $L_{\hat r}$. Returning
to $\calL^\bullet$, the cochain $c_L=\oplus c_r/L_{\hat r}\in
\calL^{d-1}$ cannot be in the image of $\calL^{d-2}\to\calL^{d-1}$ as
follows by looking at the $\hat r$-component. We have thus identified
a nonzero element $c_L$ in $H^{d-1}(\calL^\bullet)$; since this module
is isomorphic to $D_n/D_n\cdot x$ it is a simple $D_n$-module and thus $c_L$
generates it. The claim on the surjectivity of $\alpha$ then follows
since the Koszul cochain $\oplus c_i \cdot e_{\hat r}\in
\calD^{d-1}$ maps to $c_L$ under $\alpha$ and to $0$ under the
coboundary.
 
We now return to our initial assumption $d>n$ but otherwise arbitrary,
and note that, since $\alpha$ is surjective, the exact sequence
\eqref{eqn-les} simplifies to a short exact sequence
\[
0\to \image(\calK^{d-1}\to\calK^{d})\into \calK^d\to D_n/D_n\cdot x\to 0.
\]
Since the map $(\calK^{d-1}\to\calK^{d})$ is induced from
$(\calD^{d-1}\to\calD^{d})$, it follows that
$\image(\calK^{d-1}\to\calK^{d})$ is simply the ideal sum $\sum
K_{L_{\hat r}}\cdot L_r$ in $D_n$. By induction, the ideals
$K_{L_{\hat r}}\cdot L_r\subseteq D_n\cdot x$ are
generated by operators of order one. Since the cokernel $D_n/D_n\cdot
x$ of the inclusion is simple, the proof of the theorem will be
complete if we can exhibit an order one differential operator in
$\calK^d$ that is not in $D_n\cdot x$. Such operator is given by the
Euler operator $\sum x_r\del_r+d=\sum \del_r x_r-n+d$ since we are
assuming $d>n$ and so in particular $-n+d\not =0$.
\end{proof}


\subsection{Logarithmic $b$-functions}

Local systems on the open complement are an important tool for the
study of arrangements, as they connect to cohomology of Milnor fibers
(and other covers) and resonance varieties. One way to view a local
system is as a $D_n$-module of the form $D_n\cdot f^\alpha$,
$\alpha\in\CC$.  It is therefore of interest to know the structure of
such modules. 

For all locally quasi-homogeneous free divisors $f\in R_n$, the
$D_n$-annihilator of the section $1/f$ of the $D_n$-module $R_n[1/f]$
is at any point generated by derivations, see
\cite{CastroUcha-Steklov02}. The question how this compares to
$\ann_{D_n}(f^s)$ being generated by order one operators is an
interesting one.
A related property is that $b_f(s)$ have
no negative integer roots outside $s=-1$, as is the case for
hyperplane arrangements, \cite[Thm.~5.1]{W-Bernstein}. It turns out,
this condition is not sufficient, even for arrangements, see Example
\ref{exa-bracelet} below. Nonetheless, the derivations are never too
far from the annihilator for an arbitrary arrangement as we show now.

\begin{lem}
For any arrangement $\calA$ there is a finite list $Q$ of rational
numbers such that if $\alpha-\NN$ is disjoint to $Q$ then
$\ann_{D_n}(f^{\alpha}_\calA)$ is generated by derivations.
\end{lem}
\begin{proof}
We shall show inductively the following claim, which implies the
lemma. We denote $\ann_{D_n[s]}^{(1)}(f_\calA^s)$ the ideal generated
by the operators of order one in $\ann_{D_n[s]}(f_\calA^s)$.
\begin{clm} For each
logarithmic stratum $\sigma$ of $\Var(f_\calA)$ there is a polynomial
$b^\sigma_{f_\calA}(s)$ such that
\[
\left(\prod_{\tau\supseteq\sigma}b^\tau_{f_\calA}(s)\right)
\cdot\ann_{D_n[s]}(f^s_\calA)\subseteq
\ann_{D_n[s]}^{(1)}(f^s_\calA)
\text{ along $\sigma$.}
\]
\end{clm}
Granting this claim, the lemma follows with
$Q$ as the root set of $b_{f_\calA}(s)\cdot
\lcm_{\sigma}\left(b^\sigma_{f_\calA}(s)\right)$. Indeed, with $\alpha$ not a root
of this product, $\ann_{D_n}(f^s_\calA)$ evaluated at $s=\alpha$ is
inside $D_n\cdot\ann_{D_n}^{(1)}(f^\alpha)$. 
By \cite[Prop.~6.2]{Kashiwara-bfu}, if $\alpha-\NN$ is
disjoint to the root set of $b_{f_\calA}(s)$ then evaluating
$\ann_{D_n[s]}(f^s_\calA)=D_n[s](\ann_{D_n}(f^s_\calA),E-ds)$ at
$s=\alpha$ gives all of $\ann_{D_n}(f^\alpha_\calA)$ which is thus generated
by derivations.

The Claim is clear in the tame case.  By induction we can
assume that the claim holds along all positive-dimensional strata of
the logarithmic stratification of the arrangement. (During 
induction,  along $\sigma$, we may replace $f_\calA$ by
$\prod_{L_i(\sigma)=0}L_i$; this changes the defining
equation only by a unit). Let $\tilde b(s)$ be the least common
multiple of the polynomials $b^\sigma_{f_\calA}(s)$ where $\sigma$ is
positive-dimensional.  Then choose a zero-dimensional stratum $o$,
identify it with the origin of $\CC^n$, and let $E=\sum_{i=1}^n
x_i\del_i$ denote the corresponding Euler vector field.

Arguing locally, we can now assume that $f_\calA$ is central of degree
$d$ and that $\tilde
b(s)\frac{\ann_{D_n[s]}(f_\calA^s)}{\ann_{D_n[s]}^{(1)}(f_\calA^s)}$
is supported only at the origin. Since $s$ acts through the Euler
operator, it follows that this module is $D_n$-coherent and hence
holonomic. Let $\bar P$ be in the socle $\Sigma$ of this quotient
module; since the generators of $\ann_{D_n[s]}(f_\calA^s)$ and $
\ann_{D_n[s]}^{(1)}(f^s)$ can be taken to be homogeneous relative to
the grading $x\leadsto 1,\del\leadsto -1, s\leadsto 0$, we can assume
$P$ to be homogeneous. Then $\bar P$ is annihilated by $\tilde
E=\sum_{i=1}^n\del_i x_i$, while calculation reveals that also $\tilde
E\cdot \bar P=\bar P\cdot (ds+n-\deg(P))$. Since a $D_n$-module
supported at the origin is generated by its socle, $\tilde b(s)\cdot
\lcm_{\bar P\in \Sigma} (ds+n-\deg(P))$ multiplies
$\ann_{D_n}(f_\calA^s)$ into $\ann_{D_n[s]}^{(1)}(f_\calA^s)$. By
construction and Kashiwara's theorem, the roots of this polynomial are
in $\QQ$.
\end{proof}

\begin{rmk}
\begin{asparaenum}
\item The lemma and its proof is also correct for Saito-holonomic divisors
with local quasi-homogeneities $\sum a_ix_i\del_i$ that associate
non-vanishing local degrees to the defining equation.
\item Let $n=2$, $f=x^4+xy^4+y^5$. Then $\ann_{D_n[s]}(f^s)$ contains
  an operator of order two which is not multiplied into
  $\ann_{D_n[s]}^{(1)}(f^s)$ by any nonzero polynomial in $s$. Note
  that $f$ is not Euler-homogeneous.\schluss
\end{asparaenum} 
\end{rmk}

\begin{dfn}
For an arrangement $\calA$, we let $b_{\log\calA}(s)$ be the monic
minimal polynomial in $\QQ[s]$ that annihilates
$\frac{\ann_{D_n[s]}(f_\calA^s)}{\ann^{(1)}_{D_n[s]}(f_\calA^s)}$.
\end{dfn}

\begin{exa}[The bracelet]\label{exa-bracelet}
Consider the central arrangement $\calA$ in $X=\CC^4$ defined by
$f_\calA=
x_1x_2x_3(x_1+x_0)(x_2+x_0)(x_3+x_0)(x_1+x_2+x_0)(x_1+x_3+x_0)(x_2+x_3+x_0)$. The
module $\Der_{X}(-\log_0 f_\calA)$ is generated by four derivations whose
coefficients are all cubics in $x$.
The module $\Omega^1_{X}(\log_0 f_\calA)$ has also four generators,
and a minimal free resolution of the form $0\to {R_4}^1\to {R_4}^4\to
{R_4}^6\to \Omega^1_{X}(\log_0 f_\calA)\to 0$. In particular, $\calA$
is not tame: $\pdim\Omega^1(\log f_\calA)=2$.

The only non-tame point is the origin. Hence, the ideal $\calL_{f_\calA}$ is
prime in all points with at least one nonzero $x$-coordinate. However,
it has a $4$-dimensional component over the origin and in particular
is not Cohen--Macaulay.

Outside the origin, $\ann_{D_n}(f_\calA^s)$ is generated by
derivations, but this is false at the origin:
$\ann_{D_n}(f_\calA^s)$ is generated by 
$\Der_{X}(-\log_0 f_\calA)$ and a (long) operator $P$ of
order $2$ and degree $4$ (in $x$).
While we know no computer algebra system that can
compute $\ann_{D_n}(f^s_\calA)$ directly, 
one can check that
the lead term (under the order filtration) of
$P$, together with $\calL_{f_\calA}$, generates a prime ideal.
It follows from Lemma \ref{lem-know-annfs} 
that
$b_{\log \calA}(s)=s+\frac{n+\deg_{(1,-1)}(P)}{d}=s+(4+4-2)/9=s+2/3$.
\schluss
\end{exa}


\begin{rmk}
For an arrangement $\calA$ and a complex number $\alpha$, consider the
smallest $i\in\ZZ$, such that
$\ann_{D_n}(f_\calA^{\alpha-i-j})$ is generated by derivations for all
$j\in\NN$. This ``derivation index of $f^\alpha$'' should be related
to the smallest root in $\alpha+\ZZ$ of the $b$-function of $f$ on
$D_n\cdot f^\alpha$. 
\end{rmk}

From our theory, the following interesting question has a positive
answer for all divisors that fit Theorem \ref{thm-tame-annfs}, but
experimentally it is also correct for the bracelet (see Example
\ref{exa-bracelet}) and the divisor $xy(x+y)(x+zy)$:
\begin{que}
Is $\gr_{(0,1)}(\ann_{D_n}(f^s))$ always prime?
\end{que}
The only candidate for the prime ideal is the defining ideal of the
(closure in $T^*\CC^n$ of the) union (over $\alpha$) of all conormals
to $\Var(f-\alpha)$.

\subsection{Bernstein--Sato polynomials and combinatorics}

The Bernstein--Sato polynomial $b_f(s)$ of a polynomial  $f\in R_n$ is the
monic 
generator of the ideal 
\[
\langle b_f(s)\rangle = D_n[s](f,\ann_{D_n[s]}(f^s))\cap \QQ[s].
\]
It is a difficult and long-standing problem to determine the
Bernstein--Sato polynomial of an arrangement $f_\calA$. While the case
of a generic arrangement is completely understood (see
\cite{W-Bernstein, Saito-Compos07}), even in dimension three it is
very mysterious how the singularities of the arrangement contribute to
the roots of the Bernstein--Sato polynomial.

The Bernstein--Sato polynomial of an arrangement is more  constrained
than that of a general divisor. For example, all its roots have to be
in the interval $(-2,0)$, and one can read off the possible denominators of
the roots  directly from the arrangement by \cite[Thms.~1+2]{Saito-bfu}. 
A natural question in the context of arrangements is to what extent
the intersection lattice governs the Bernstein--Sato polynomial. The
following example indicates that one ought to ask a more refined
question. 

\begin{exa}\label{exa-ziegler}
Let $P_1,\ldots,P_6$ be six points in $\PP^2$ and let $\calA$ be the
arrangement of 9 planes in $\CC^3$ given as the union of the
hyperplanes $G_{i,j}$ passing through $P_i$ and $P_j$ for
$i-j=1\modulo 6$ or $i-j=3\modulo 6$.  The arrangement is sketched as
the union of the solid and dashed (but not the dotted) lines below. We assume
that the points are sufficiently generic to make all 9 lines distinct,
and so that the only triple points are the $P_i$.
Let $J=\Jac(f_\calA)$.

Note that given three homogeneous linear forms
$L_1,L_2,L_3=aL_1+bL_2\in\CC[x,y]$ then a curve defined by $\eps(x,y)$
satisfies $\eps\cdot L_3\in \Jac(L_1\cdot L_2\cdot L_3)$ at $(0,0)$ if
and only if $\eps(0,0)=0$ and the tangent of $\eps$ at $(0,0)$ is the
line defined by $aL_1-bL_2$. We call this tangent line the \emph{conjugate
  of $L_3$ relative to $L_1$ and $L_2$}.
\smallskip

\newlength{\currentparindent}
\setlength{\currentparindent}{\parindent}
\noindent\begin{minipage}{\textwidth}
\begin{minipage}{0.7\textwidth}
\setlength{\parindent}{\currentparindent} For generic choices of the
points, $H^0_\frakm(R_3/J)$ is generated in degree $9$ by six forms
(one for each vertex) 
given as the union of the solid and dotted (but not the dashed) lines
indicated on the right.  The ``tangential'' dotted line on the far
left is the conjugate of
the big diagonal relative to the dashed lines at 
the vertex in question (and
hence multiplies the big diagonal into the Jacobian there).
\end{minipage}\hfill
\begin{minipage}{0.27\textwidth}
\includegraphics[width=\textwidth]{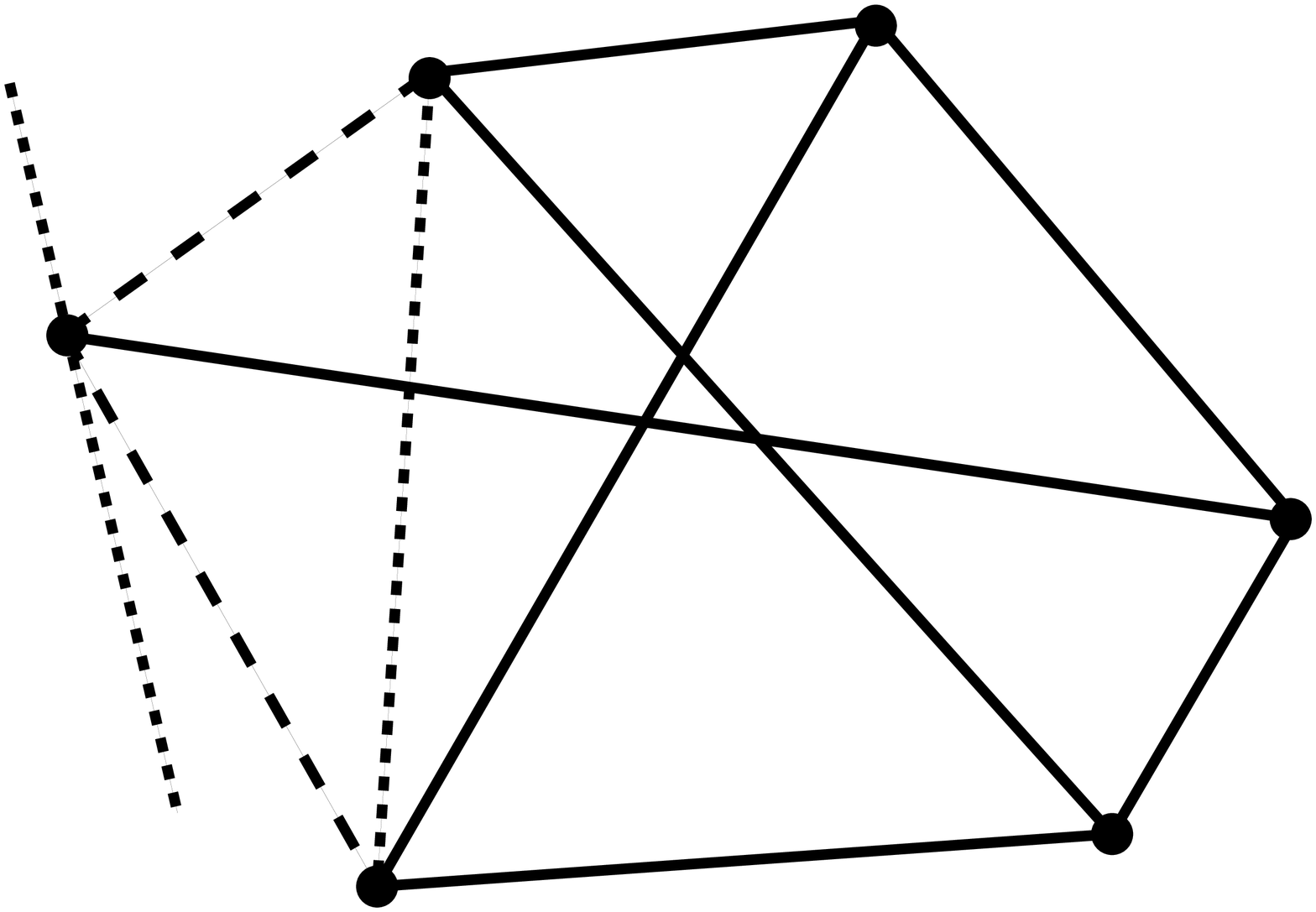}
\end{minipage}
\end{minipage}


We investigate the question under what circumstances
$H^0_\frakm(R_3/J)$ contains the coset $g_8$ of the product of a cubic
form $c$ with $G_{1,6}\cdot G_{5,6}\cdot G_{4,5}\cdot G_{2,5}\cdot
G_{3,6}$. Such curve must:
\begin{enumerate}
\item pass through the points $P_7=G_{1,2}\cap G_{3,4}$ and
  $P_8=G_{1,4}\cap G_{2,3}$;
\item pass through the points $P_1,P_2,P_3,P_4$ and in each of them
  have the correct tangent. (For example, the tangent at $P_1$ must be
  the conjugate line of $G_{6,1}$ relative to $G_{1,2}$ and $G_{1,4}$).
\end{enumerate}
The vector space of cubics that passes through $P_1,P_2,P_3,P_4,P_7,P_8$ is
typically $4$-dimensional. If $P_5$ and $P_6$ are in generic position, the
tangent conditions at $P_1,P_2,P_3,P_4$ on $c$ are independent and so
no suitable $c$ exists.

Suppose $P_1,\ldots,P_5$ are in generic position and fixed.  Then
$P_5$ determines the necessary tangents of $c$ at $P_2$ and $P_4$. View $P_6$
as movable, and choose the tangent direction for $c$ at $P_3$ (hence a
line on which $P_6$ should be located).
This then determines a cubic with a tangent at $P_1$
that we cannot control.  The conjugates of the tangents of $c$ at
$P_1$ (relative to $G_{1,2}$ and $G_{1,4}$) and $P_3$ (relative to
$G_{2,3}$ and $G_{3,4}$) yield a unique point $P_6$ for which the
conditions (1) and (2) above are satisfied. 

Hence, with $P_1,\ldots,P_5$ fixed, the choices for $P_6$ that make
the conditions (1) and (2) solvable form a curve $q$, parameterized by
the choice of a line through $P_3$, and this curve
meets any line through $P_3$ in exactly one point other than
$P_3$. Choosing lines through $P_3$ that pass nearby $P_1$ one sees
that $P_1$ lies on $q$ and $q$ is smooth there, with tangent given by
the conjugate to the line through $P_1$ and $P_3$ relative to
$G_{1,2},G_{1,4}$. By symmetry, a similar statement holds at $P_3$.

Now pick a $P_6$ on $q$ but generic otherwise, keep $P_1,P_2,P_3,P_4$
fixed, and move $P_5$ on $q$. The cubic $c$ can then not change, since
$P_1,\ldots,P_4,P_6$ pin it down. By the same reasoning as before, $q$
passes through $P_2,P_4$ and is smooth in both points.

Finally, keep $P_1,\ldots,P_5$ fixed and degenerate $P_6$ into 
$P_5$. In order to have an element $G_{1,6}\cdot G_{5,6}\cdot
G_{4,5}\cdot G_{2,5}\cdot G_{3,6}$ times a cubic in
$H^0_\frakm(R_3/J)$ it is necessary that $c\cdot G_{5,6}$ should
\begin{enumerate}
\item pass through the points $P_7=G_{1,2}\cap G_{3,4}$ and
  $P_8=G_{1,4}\cap G_{2,3}$;
\item pass through the points $P_1,P_2,P_3,P_4$ and in each of them
  have the correct tangent.
\end{enumerate}
Because of the degeneration, the direction of
$G_{5,6}$ can be chosen freely and this added degree of freedom makes
the problem solvable.

\emph{In summa}, for generic choices of $P_1,\ldots,P_5$ there is an
element $g_8=G_{1,6}\cdot G_{5,6}\cdot G_{4,5}\cdot G_{2,5}\cdot
G_{3,6}\cdot c$, $c$ a cubic, in $H^0_\frakm(R_3/J)$ if and only if
$P_6$ is a generic point on a curve $q$ which satisfies: 
$P_1,\ldots,P_5\in q$; $q$ meets a generic line through $P_3$ in exactly one
other point; $q$ is smooth at $P_3$. In other words, $q$ is the unique
quadric through $P_1,\ldots,P_5$.

To be explicit, one can verify with {\sl Macaulay 2} that one may take
$G_{1,2}=2x+y+z$, $G_{2,3}=x+y+z$, $G_{3,4}=2x + 3y + 4z$,
$G_{4,5}=z$, $G_{5,6}=x + 3z$, $G_{6,1}=y$, $G_{1,4}=2x + 3y + z$,
$G_{2,5}=x$ and $G_{3,6}=x + 2y + 3z$. Then $q=2x^2+3xy+7xz+3yz+3z^2$
and
$c=20x^3+68x^2y+73xy^2+24y^3+60x^2z+130xyz+65y^2z+51xz^2+54yz^2+13z^3$.

Pellikaan, and Van Straten and Warmt proved
\cite[Thm.~4.7]{vanStratenWarmt} that since $\dim(J)=1$,
$\deg(H^0_\frakm(R_3/J))$ is an Artinian Gorenstein module.  The proof
identifies $R_3$ and $R_3/J$ with their canonical modules and also, in
\cite[Prop.~3.6]{vanStratenWarmt}, the Koszul complex on the
derivatives of $f_\calA$ with its dual complex. Similar ideas are
detailed in \cite{DimcaSaito-Koszul1212.1081}. In the graded case,
these identifications incur appropriate shifts and one obtains
symmetry of the degrees of $H^0_\frakm(R_3/J)$ about
$3\deg(f)/2-3$. Hence a generator $g_8$ in degree 8 corresponds to a
socle element in degree $2(3\cdot 9/2-3)-8=13$. (The Hilbert series of
$H^0_\frakm(R_3/J)$ in the example above is
$T^8+4T^9+6T^{10}+6T^{11}+4T^{12}+T^{13}$; in the generic case, when
the six points are not on a quadric, it is
$4T^9+6T^{10}+6T^{11}+4T^{12}$.)

Since the singular loci of the arrangements are 1-dimensional, and
since further the critical root $-16/9$ has no chance to be root of a
local Bernstein--Sato polynomial outside the origin in any case, it
follows from \cite{Saito-Compos07} that $-16/9$ is a root of
$b_{f_\calA}(s)$ if and only if $16/9$ contributes to the pole order
spectrum at the origin. However, by \cite{DimcaSaito-Koszul1212.1081},
one can decide this directly from the Hilbert series of
$H^0_\frakm(R_3/J)$, that of $R_3/J$ and that of the Koszul homology
groups of the derivatives $f_1,f_2,f_3$ on $R_3$.  In particular,
computations with \emph{Macaulay 2} \cite{M2} can now easily confirm
that in the degenerate case $-16/9$ must be a root of the
Bernstein--Sato polynomial while it is not in the generic case.  A
detailed numerical discussion and an explanation of this conclusion
has been given in \cite[Remark 4.14]{Saito-bfu} where the
Bernstein--Sato polynomial has been determined entirely, but compare
also \cite[Remark 5.4]{Saito-bfu}.

An alternative computer-based method is the following.  The
arrangements are tame, no matter whether the vertices lie on a
quadric.  Thus, $\ann_{D_n[s]}(f_\calA^s)$ is generated by the Euler
relation and the logarithmic vector fields that kill $f_\calA$. By
considering the $D_n[s]$-ideal generated by $E-9s,s+16/9$ and
$\Der(-\log_0 f_\calA)$, V.~Levandovskyy has confirmed with PLURAL
\cite{Plural} that the Bernstein--Sato polynomials differ in the
generic and in the degenerate case.

Since the intersection lattice of $\calA$ is the same for degenerate
and generic $\calA$, the Bernstein--Sato polynomial of an arrangement
is not a function of the intersection lattice: differential invariants
remember finer structure.
\schluss
\end{exa}

It would seem natural to conclude that for an arrangement the
Bernstein--Sato polynomial is not determined by linear information.
However, it is our opinion that one should refine the
intersection lattice and stick to linear data. 

\begin{dfn}
For a central arrangement $\calA$, we define the \emph{syzygetic
  intersection lattice} $L^+_\calA$ as follows. Let $L^0_\calA=L_\calA$
and for $i\geq 1$ set $L^i_\calA$ to be the set of all vector spaces
that appear as nontrivial intersections of nontrivial sums of elements
of $L^{i-1}_\calA$. Now let $S$ consist of the \emph{syzygetic}
elements $V$ in $\bigcup_i L^i_\calA$ that satisfy: $V\in L^i_\calA$
is in $S$ if there are elements $V_1,\ldots,V_k$ in $L^{i-1}_\calA$
such that $V=\sum_j V_j$ but $\dim(V)<\sum_j\dim(V_j)$. Then define
$L^+_\calA$ as the pair $(L_\calA,S)$.
\end{dfn}
The component $S$ of $L^+_\calA$ indicates syzygies (in the original
sense of the word) between elements of $L_\calA$.

In the degenerate case of Example \ref{exa-ziegler}, by the theorem
of Pappus as refined by Pascal, the points that arise (say) as the
intersections $G_{1,5}\cap G_{4,2}$, $G_{1,6}\cap G_{3,4}$,
$G_{2,6}\cap G_{3,5}$ are \emph{collinear}. (There are a maximum of 60
such Pascal lines in the degenerate case, but some of them may
coalesce for certain arrangements). Thus, in the degenerate case, the
affine cones over all Pascal lines would show up in $L^+(\calA)$.  In
contrast, the syzygetic lattice of all arrangements in the generic
case contains no Pascal lines, so that $L^+_\calA$ discriminates
between the generic and the degenerate case although $L_\calA$ does
not.

\subsection{Strong Monodromy Conjecture}

Often, even very concrete questions regarding the root set $\rho_f$
of the Bernstein--Sato polynomial of a polynomial are very
hard to answer, even for arrangements. For example, the following
seems essentially open:

\begin{que}\label{que-n/d}
  Suppose $f$ is a homogeneous polynomial of degree $d$ of embedding
  dimension $n$.  Assuming that $f$ is not the product of divisors in
  different sets of variables, under what circumstances is the number
  $-n/d$ a root of $b_f(s)$?
\end{que}

The $n/d$-Conjecture \ref{cnj-n/d} of Budur, \mustata\ and Teitler 
\cite{BudurMustataTeitler-GeomDed11} lists central indecomposable
hyperplane arrangements as one important case where 
Question \ref{que-n/d} is expected to have a positive answer.  See
\cite{BudurSaitoYuzvinsky-JLMS11} for positive results, including the
case of reduced arrangements in dimension $3$.

A sufficient condition for the $n/d$-Conjecture to hold is that the
(coset of the) element $1\in R_n$ be nonzero in the vector space
$U_f\cong H^{n-1}(M_f,\CC)$,
see \cite[Thm.~4.12]{W-Bernstein}. We do not know of significant
classes of singularities where this topological version of the
question has been answered.  Our structural results about
$\ann_{\calD_X}(f^s)$ now enable us to prove that most arrangements
satisfy Conjecture \ref{cnj-n/d}.


\begin{thm}\label{thm-n/d}
  Let $\calA$ be a central indecomposable arrangement (reduced or
  otherwise) with defining equation $f_\calA=\prod_1^d L_i\in R_n$ on
  $X=\CC^n$. Then $\Der_{X}(-\log_0 f_\calA)$ is contained in the
  ideal $D_n\cdot x$.

 If, in addition, $\ann_{D_n}(f_\calA^s)$ is generated by derivations
then $b_{f_\calA}(-n/d)=0$.
\end{thm}
\begin{proof}
  As $f_\calA$ is homogeneous, $\Der_{X}(-\log_0 f_\calA)$ and
  $\ann_{D_n}(f_\calA^s)$ are graded modules, where the degree of each
  $x_i$ is $1$ and that of each $\del_i$ is $-1$.  The fact that
  $f_\calA$ is indecomposable implies that there are no homogeneous
  logarithmic derivations of negative degree (that is, with constant
  coefficients). Namely, if $\delta=\sum_1^n c_i\del_i$ were such a
  negative degree logarithmic derivation with $c_i\in\CC$, then a
  judicious change of coordinates transforms $\delta$ into
  $\delta'=\del_1$. But if this were to annihilate the transformed
  $\calA$, then $\calA$ should not be essential, and in particular
  decomposable.

We may restrict attention to homogeneous derivations.  Observe that
derivations of positive degree are always in $D_n\cdot x$ since in
that case the degree of the coefficient exceeds the order of the
operator. It follows that we only need to consider derivations
$\delta$ that annihilate $f$ and are of degree zero.

  Such $\delta$ can be written as $\delta=x^TB\del$ for a suitable
  square matrix $B\in\CC^{n,n}$.  Restricting its domain, one may interpret
  $\delta$ also as a linear operator on $\bigoplus_1^n\CC \cdot
  x_i=:V$. Note that
  \begin{eqnarray*}
    \underbrace{\sum x_i
      b_{i,j}\del_j}_{x^TB\del}\bullet(\underbrace{\sum c_k
      x_k}_{x^Tc})&=&\sum x_ib_{i,j}\delta_{j,k}c_k=\sum
    x_ib_{i,j}c_j=x^TBc
  \end{eqnarray*}
  so that $B$ represents both $\delta$ and this linear transformation
  on $V$.

  As $\delta\bullet(f_\calA)=0$, $\delta$ is logarithmic along all
  $L_k$ as well, so $\delta\bullet(L_k)=\gamma_kL_k$ with
  $\gamma_k\in\CC$. In particular, the sum of the eigenspaces of
  $\delta$ on $V$ is all of $V$, as $\calA$ is essential. It follows
  that $B$ is diagonalizable. Since a coordinate change $x'=Ax$ on
  $\CC^n$ incurs the coordinate change $\del_{x'}=A^{-T}\del_x$ on
  $T^*(\CC^n)_0$, it follows that diagonalizing $B$ can be achieved by
  an appropriate coordinate change in $\CC^n$. 

  In the right coordinates $\delta=\sum_1^n w_ix_i\del_i$. Then,
  however, the condition $\delta\bullet(f_\calA)=0$ is a non-standard
  quasi-homogeneity of $f$ in the sense that not all $w_i$ can be
  equal. We show that this can't be happening.
  \begin{clm}
    An arrangement with a non-standard quasi-homogeneity $\delta$ is
    decomposable.
  \end{clm}
  \begin{proof}\pushQED{\qed}
    In a $\ZZ$-graded domain, the product of two nonzero expressions is
    $\delta$-homogeneous precisely if the factors are. It follows that
    each defining hyperplane of the arrangement inherits the
    homogeneity $\delta$. But a linear form is $\delta$-homogeneous if
    and only if all variables occurring in that form have the same
    $\delta$-degree. We have $w_i=\delta(x_i)/x_i\in\CC$. Since $\delta$
    is non-standard, the decomposition
    \[
    \{x_1,\ldots,x_n\}=\{x_i\mid
    w_i=w_1\}\sqcup\{x_j\mid w_i\neq w_1\}
    \]
    is nontrivial; the induced factorization shows that
    $\calA$ is decomposable.\qedhere$_{\rm Claim}$
  \end{proof}

  Returning to the proof of the theorem, indecomposable arrangements
  are not annihilated by logarithmic derivations of degree zero and it
  follows that $\ann_{D_n}(f^s_\calA)\subseteq D_n\cdot x$.

  Under the assumptions for the second claim we have proved thus the
  existence of a map 
\[
\frac{D_n}{D_n\cdot\Der_{X}(-\log_0 f_\calA)}\stackrel{=}{\to}
\frac{D_n}{\ann_{D_n}(f_\calA^s)} \onto 
\frac{D_n[s]}{D_n[s]\cdot (x,E-ds)}\cong
\frac{D_n}{D_n\cdot x}.
\]  
  Pick a coset $\bar P=P+D[s]\cdot (x,E-ds)$ in the target of this
  map. Since $D_n[s]\cdot (x,E-ds)$ is $(1,-1)$-graded we may assume that
  $P$ is $(1,-1)$-homogeneous.  We calculate $s\bar
  P=\bar{Ps}=\bar{PE/d}=-\bar{nP/d}$ since $E+n\in D_n\cdot x$.  It
  follows that $s$ has minimal polynomial $s+n/d$ on the (nonzero!)
  quotient $D_n[s]/D_n[s]\cdot (x,E-ds)$ of
  $D_n[s]/D_n[s]\cdot(f_\calA,E-ds,\ann_{D_n}(f_\calA^s))$. So the
  Bernstein--Sato polynomial of $\calA$ has to be a multiple of
  $(n+ds)$.
\end{proof}

Recall the Strong Monodromy Conjecture from the introduction. 
For convenience, we introduce the following abbreviation from
\cite{Torrelli-BullMathSoc04}. 

\begin{ntn}
If $\ann_{\calD_X}(f^s)$ is generated by derivations, we say that
\emph{$f$ satisfies condition $(A_s)$}.
\end{ntn}

\begin{cor}
An arrangement (reduced or otherwise) with property $(A_s)$
satisfies the Strong Monodromy Conjecture. Specifically,
this is true for tame arrangements.
\end{cor}
\begin{proof}
Arrangements allow for a combinatorial resolution of singularities,
which makes the computation of the candidate poles of the topological
zeta-function straightforward. This was used in
\cite{BudurMustataTeitler-GeomDed11} to show that the Strong Monodromy
Conjecture for reduced arrangements boils down to the $n/d$-Conjecture
for reduced arrangements.  A closer inspection shows that this
specific statement does a) not require reducedness of the divisor, and
b) actually shows that for an arrangement $\calA$ (reduced or
otherwise) to satisfy the Strong Monodromy Conjecture it is sufficient
to know that each indecomposable full subarrangement of $\calA$
satisfies the $n/d$-Conjecture. Here we say that a subarrangement
$\calA'$ is \emph{full} if for some flat of $\calA$ the arrangement
$\calA'$ consists of exactly the hyperplanes passing through the flat,
and the multiplicity of each hyperplane occurring in $\calA'$ is its
multiplicity in $\calA$.

The property $(A_s)$, as well as tameness, is inherited from any
arrangement $\calA$ with this property to each of its full
subarrangements, as one sees from localizing at the respective
flat. Suppose $\calA$ has $(A_s)$.
By Theorem \ref{thm-n/d}, all its indecomposable
full subarrangements satisfy the $n/d$-Conjecture and so by the
previous paragraph $\calA$ satisfies the Strong Monodromy Conjecture.

The last claim of the corollary follows then from
Theorem \ref{thm-generic-Terao}. 

\end{proof}

\begin{rmk}
\begin{asparaenum}
\item By the previous result, all arrangements in dimension three
  satisfy the Strong Monodromy Conjecture. For reduced arrangements,
  this was shown in \cite{BudurSaitoYuzvinsky-JLMS11}. (In fact,
  \cite{Saito-bfu} proves an even stronger form for reduced rank three
  arrangements: the product of $Z_f(s)\cdot b_f(s)$ is a
  polynomial). Their approach to showing that the $n/d$-Conjecture
  holds in the requisite cases is totally different from ours. We do
  not know how to generalize their method to the non-reduced
  situation.
\item Using {\sl Macaulay2} one verifies (by showing that the ``long
  operator'' in Example \ref{exa-bracelet} is in $D_n\cdot x$) that the
  bracelet satisfies the $n/d$-Conjecture and hence the Strong
  Monodromy Conjecture, even though it fails $(A_s)$.
\item Consider $f_\calA=(x^3+2y^3)(z^8-w^8)(x+z+w)$.  Veys informs us
  that this arrangement does not have $-n/d=-1/3$ as pole in the
  topological zeta function, and $-1/3$ can also not be created as
  pole of more general integrals in the sense of
  \cite{Bories-complutense}. Nonetheless, $f_\calA$ is tame and so
  $\ann_{D_n}(f_\calA^s)$ is generated by derivations wherefore the
  $n/d$-conjecture holds, and so $-1/3$ is a root of the
  Bernstein--Sato polynomial.
\schluss\end{asparaenum}
\end{rmk}

\section{Appendix: Logarithmic derivations under blow-ups}\label{sec-blow-up}

\begin{ntn}\label{ntn-blow-up}
Let $X$ be a smooth scheme over $\CC$, with subschemes
$C\subseteq Y\subseteq X$. Assume that $C$ is smooth and denote by
$\calI_C\supseteq \calI_Y$ the ideal sheaves.
We say that a derivation on $X$ is \emph{logarithmic} along a
subscheme $V$ defined by $\calI_V$ if $\delta(\calI_V)\subseteq
\calI_V$. 

Let
\[
\pi\colon X'\to X
\]
be the (smooth) blow-up of $X$ at $C$ and denote $Y'\subseteq X'$ the
total transform of $Y$.
\end{ntn}

\begin{lem}\label{lem-res-omega}
In the context above, 
$\pi_*(\Omega^i_{X'}(\log \pi^*(Y)))=\Omega^i_X(\log Y)$.
\end{lem}
\begin{proof}
The statement is local in the base, so we may assume that $X$ is
smooth and affine, and represent $\Omega^i_X(\log Y)$ by their modules
of global sections. We may assume further that $\codim_X(C)\geq 2$ as
otherwise $\pi$ is an isomorphism.  Set $R=\Gamma(X,\calO_X)$
and $I_C=\Gamma(X,\calI_C)$.

Pulling back differentials along $\pi$ is a linear faithful functor
and there are $\calO_X$-module inclusions
$\underbrace{\Omega^i_X(\log Y)}_{=:M}\subseteq
\underbrace{\Gamma(X',\Omega^i_{X'}(\log Y'))}_{=:M'}\subseteq
\underbrace{K\otimes_R\Omega^i_X}_{=:M''}$, $K$ denoting the field of
fractions of $R$.

Let $Q$ be defined by the exactness of $0\to M\to M'\to Q\to 0$. Since
$M''$, and hence also $M'$, is torsion-free, the
associated long exact sequence of local cohomology gives an injection
$H^0_{I_C}(Q)\into H^1_{I_C}(M)$. Now $M$ is a second syzygy, say the
kernel of the map $F_1\to F_0$ between free modules. Then if $L$ is
the image of this map, long exact sequences again show that
$H^1_{I_C}(M)=H^0_{I_C}(L)$ which vanishes as $L$ sits in a free
module and $I_C$ has height at least two. We have shown that
$H^0_{I_C}(Q)=0$ but since $Q$ is 
supported in $C$ this forces $Q=0$.
\end{proof}

We now consider lifting logarithmic derivations along a blow-up. While
logarithmic derivations and logarithmic $(n-1)$-forms can be
identified locally, globally
$\Omega^{n-1}_X(\log Y)$ is really $\Der_X(-\log
Y)\otimes\calO_X(Y)\otimes \Omega^n_Y$ and so exhibits different
behavior: the push-forward of the logarithmic derivations along
$\pi^*(Y)$ is, in general, not all of the logarithmic derivations
along $Y$.

\begin{thm}\label{thm-res-der}
In the setting of Notation \ref{ntn-blow-up},
\begin{asparaenum}
\item  an element of $\Der_X$ lifts to $X'$ if and only if it is logarithmic
along $\calI_C$; we have $\pi_*(\Der_{X'})=\Der_X(-\log \calI_C)$;
\item 
if $Y$ is $C$-saturated,
$\calI_Y:_{\calO_X}\calI_C=\calI_Y$, 
then 
\[
\pi_*(\Der_{X'}(-\log \pi^*(Y)))=\Der_X(-\log
  C)\cap \Der_X(-\log Y). 
\]
\end{asparaenum}
\end{thm}
\begin{proof}
Again, we assume that $X$ is
smooth and affine,  and that $\codim_X(C)\geq 2$.
Set $R=\Gamma(X,\calO_X)$ and write $C=\Var(I_C)$.

By shrinking $X$ we can assume that $k=\codim_X(C)$ and
$I_C=(f_1,\ldots,f_k)R\subseteq R$ where the $f_i$ form a regular
sequence in any order.  Let $\delta\in\Der_X$ be logarithmic along
$C$. Since $\delta\bullet(I_C)\subseteq I_C$, it induces a derivation
$\tilde \delta$ of $t$-degree zero on the Rees ring
\[
\calR=\calR(R,I_C)=\bigoplus_{i\in\NN} (I_Ct)^i.
\]
Derivations of $t$-degree zero induce derivations on each homogeneous
localization $\calR[(tf_i)^{-1}]$, stabilize the degree-zero part, and
then obviously agree on overlaps.  Since $X'=\Proj(\calR)$, $\tilde
\delta$ induces a global derivation $\delta'=:\pi^*(\delta)$ on
$\calO_{X'}$. The exceptional divisor $E_\pi$ is defined by the ideal
sheaf $I_Ct\calR$; as $\tilde\delta\bullet(I_Ct\calR)\subseteq I_Ct\calR$,
$\delta'$ is logarithmic along $E_\pi$. The $R$-morphism
$\delta\mapsto \pi^*(\delta)$ from $\Der_X$ to $\Der_{X'}(-\log
E_\pi)$ is injective, since $X$ and $X'$ are smooth and agree outside
$C$.

Now suppose additionally that $\delta\in\Der_X(-\log Y)$; then $\tilde
\delta$ preserves the extension of $I_Y$ to $\calR$, and hence
$\delta'$ is logarithmic along the total transform of $Y$.  Since
$I_C$ is prime and $I_Y:_R I_C=I_Y$, a dense open set of each component
of the proper transform $Y'$ is outside $E_\pi$. It follows that 
$\delta'$ is
also logarithmic along $Y'$ and so $\Der_X(-\log Y)\cap \Der_X(-\log
C)$ is a submodule of $\Gamma(X',\Der_{X'}(-Y'))$ and of
$\Gamma(X',\Der_{X'}(-\pi^*(Y)))$.


Conversely, let $\delta'$ be a global derivation on $X'$ and let $U_i$
be the standard open set $\Spec(\calR_i)$ in $X'$ where
$\calR_i=(R[tf_1,\ldots,tf_k,(tf_i^{-1})])_0=R[f_1/f_i,\ldots,f_k/f_i]$.
By definition, $\delta'$ induces a derivation on each $\calR_i$ which
we also denote $\delta'$.

Since $R$ is a domain, $R\subseteq \calR_i \subseteq R[1/f_i]$ and so
$\delta'\bullet(R)\subseteq\bigcap_i\calR_i\subseteq\bigcap_i
R[1/f_i]$.  The latter is the ideal transform of $R$ with respect to
$I_C$, and since the depth of $I_C$ on $R$ is at least $2$ we have
$(\bigcap R[1/f_i])/R=H^1_I(R)=0$.  Hence $\delta'$ is a derivation on
$R$ which we denote $\delta$. We must show that it is logarithmic
along the center of the blow-up. Note that $\pi^*(\delta)=\delta'$ again.

\begin{lem}\label{lem-regSeqDerivation}
Let $R$ be a domain, $f_1,\ldots,f_k,g$ a regular sequence in $R$, and
$\delta'$ a derivation on both $R$ and $R[f_1/g,\ldots,f_k/g]$. Let
$I$ be the $R$-ideal generated by $f_1,\ldots,f_k$. Then
$\delta'\bullet(I)\subseteq I+Rg$.
\end{lem}
\begin{proof}
 By hypothesis,
 $\delta'\bullet(f_i/g)=\delta'\bullet(f_i)/g-f_i\delta'\bullet(g)/g^2$
 can be written as $\sum_{j=0}^rP_j(f_1/g,\ldots,f_k/g)$ for suitable
 homogeneous polynomials $P_0,\ldots,P_r\in R[x_1,\ldots,x_k]$, $P_j$
 being of degree $j$. Choose such presentation with $r$ minimal.

We show first that one can assume $r\le 2$. Indeed, if $r>2$ clear
denominators to see that $P_r(f_1,\ldots,f_k)\in Rg$. By
Lemma~\ref{lemma2} below, $P_r(f_1,\ldots,f_k)=Q_r(f_1,\ldots,f_k)$
for a suitable homogeneous polynomial $Q_r(x)=g\sum_M q_Mx^M\in
gR[x_1,\ldots,x_k]$ using multi-index notation $x^M=\prod_ix_i^{m_i}$
with $|M|=\deg(P_r)=r$. Then, abbreviating $f_1,\ldots,f_k$ to $f$,
and $f_1/g,\ldots,f_k/g$ to
$f/g$,
\begin{eqnarray*}
P_r(f/g)=\sum_{M}q_M\cdot g\cdot (f/g)^M
=\sum_{M}q_M\cdot g^{1-|M|}\cdot f^M.
\end{eqnarray*}
For each multi-index $M$, let $i(M)$ be some index with $M_i>0$ and
denote $e_i$ the unit vector in direction $i$. Then
$\sum_{M}q_Mg^{1-|M|}f^M=\sum_{M}q_Mf_{i(M)}g^{1-|M|}f^{M-e_{i(M)}}$
can be viewed as evaluation of $\sum_{M}q_Mf_{i(M)}x^{M-e_{i(M)}}$ at
$f/g$. The latter is a polynomial of degree $r-1$. And so the
presentation $\sum_{j=0}^rP_i(f_1/g,\ldots,f_k/g)$ for
$\delta'(f_i/g)$ didn't use minimal
$r$.

Hence, we may assume that $r\le 2$.  In that case, we obtain, clearing
denominators,
\[
\delta'\bullet(f_i)g-\delta'\bullet(g)f_i=P_0(f)g^2+P_1(f)g+P_2(f).
\]
In particular,
$\delta'\bullet(f_i)g-P_0(f)g^2=P_1(f)g+P_2(f)+\delta'\bullet(g)f_i\in
I$. Since $g$ is regular on $R/I$,
$\delta'\bullet(f_i)-P_0(f)g\in I$ and so $\delta'\bullet(f_i)\in
R(f,g)$. This finishes the proof of Lemma \ref{lem-regSeqDerivation}
\end{proof}

Now return to the proof of the Theorem. Lemma
\ref{lem-regSeqDerivation} implies, looking at $U_i$, that
$\delta\bullet (f_i)\in I$ for each $i$. Hence, $\delta$ is
logarithmic along $I_C$, and therefore $\delta'$ is logarithmic along
$E_\pi$.

If $\delta'$ is logarithmic along $Y'$ then $\delta'$ preserves each
$I_Y\cdot \calR_i$, and hence $\delta$ sends $I_Y$ into
$\bigcap(I_Y\cdot \calR_i)$. However,
$\bigcap(I_Y\cdot\calR_i)/I_Y=H^1_{I_C}(I_Y)$. Since $C$ is cut out by
a regular sequence of length at least $2$ in $R$,
$H^1_{I_C}(I_Y)=H^0_{I_C}(R/I_Y)$ and this last module vanishes since
by hypothesis $R$ is $\calI_C$-torsion-free. It follows that
$\delta$ is logarithmic along $Y$ and the theorem follows. 
\end{proof}

\begin{lem}
\label{lemma2}
Let $R$ be a domain, and let $f_1,\ldots,f_m,g_1,\ldots,g_{m'}$ be a
regular sequence in every order on $R$. Write $f=f_1,\ldots,f_m$ and 
$g=g_1,\ldots,g_{m'}$.  

Suppose $P(x)\in R[x_1,\ldots,x_{m}]$ is homogeneous of degree $k$ and 
satisfies $P(f)\in Rg$. Then there exists $Q\in R[x_1,\ldots,x_m]$,
homogeneous of degree $k$,
with $P(f)=Q(f)$ and $Q(x)\in gR[x]$.
\end{lem}
\begin{proof}
\begin{asparaenum}
\item If $m=1$ then we have $rf^k\in Rg$ and so regularity implies that the
  lemma holds in that case.
\item If $P$ is linear, $\sum r_if_i\in Rg$ and so $r_m\in
  R(g,f_{<m})$ can be written as $\sum
  b_jg_j+\sum_{i<m}a_if_i$. Rewriting, we obtain $P(f)=\sum_{<m}
  r_if_i+r_mf_m=\sum_{<m}r_if_i+\sum_{<m} f_ma_if_i+ \sum
  f_mb_jg_j=\sum_{i<m}(f_ma_i+r_i)f_i+\sum f_mb_jg_j\in
  R(f_{<m,g})$. By induction on $m$, there is a linear polynomial
  $Q'(x_1\ldots,x_{m-1})$ over
  $R$ with coefficients in $Rg$ which when evaluated at $f_{<m}$
  yields $\sum_{i<m}(f_ma_i+r_i)f_i$. But then
  $P(f)=Q'(f_{<m})+f_m\sum b_jg_j$ and the lemma follows in this case.
\item The general case. We use induction on $m$. Consider the
  stratification of the monomials in $m$ variables of degree $k$ into
  the following subsets: $S_m=\{x_m^k\}$, $S_j=\{\text{monomials in
    $x_{\geq j}$}\}\smallsetminus S_{j+1}$ for $0<j<n$. So $S_j$ is
  the set of monomials in $R[x_j,\ldots,x_m]$ divisible $x_j$. We are
  going to construct a sequence of polynomial identities
\[
(j)\colon\qquad P^{(j)}(x)=\underbrace{\sum_{\ell<j}P^{(j)}_\ell(x)}_
{P^{(j)}_{<j}(x)}+P^{(j)}_j(x)+\underbrace{\sum_{\ell>j}P^{(j)}_\ell(x)}_
{P^{(j)}_{>j}(x)}
\]
such that the following hold: 
$P^{(j)}_\ell(x)$ is a sum of monomials indexed by $S_\ell$; 
$P^{(j)}_{>j}(x)\in gR[x_{j+1},\ldots,x_m]$; 
$P^{(j)}(f)=P(f)$ for all $j$.

For $j=m$, take for $P^{(m)}_m(x)$ the part of $P(x)$ indexed by
$S_m$, and put $P^{(m)}_{>m}(x)=0$ and
$P^{(m)}_{<m}(x)=P(x)-P^{(m)}_m(x)$. 

For any $j$,  $P^{(j)}_{<j}(x)$ is
automatically in the ideal generated by $x_{<j}$, so evaluation at $x=f$
gives $P^{(j)}_j(f)=P(f)-P^{(j)}_{<j}(f)-P^{(j)}_{>j}(f)\in R(g,f_{<j})$.

Given identity $(j)$ we now construct identity $(j-1)$ with the
corresponding properties.  As $x_j$ divides $P^{(j)}_j(x)$,
$P^{(j)}_j(f)/f_j\in R$.  Then $R(g,f_{<j})\ni P^{(j)}_j(f)=f_j\cdot(
P^{(j)}_j(f)/f_j)$ implies by regularity that $P^{(j)}_j(f)/f_j\in
R(g,f_{<j})$. As the polynomial $P^{(j)}_j(x)/x_j\in R[x_{\geq j}]$ is
homogeneous of degree $k-1$, the inductive hypothesis asserts the
existence of a homogeneous polynomial of degree $k-1$ in
$x_j,\ldots,x_m$ with coefficients in $R(g,f_{<j})$ and which
evaluates at $f$ to $P^{(j)}_j(f)/f_j$. Explicitly, $P^{(j)}_j(f)$ is the value
at $x=f$ of the sum of one polynomial $P^{(j)}_{j,g}(x)\in gR[x]$ and
another polynomial in $f_{<j}R[x]$, both homogeneous of degree
$k-1$. The latter polynomial can be changed (without affecting its
value at $x\leadsto f$) into a polynomial
$P^{(j)}_{j,<}(x)\in x_{<j}R[x]$ which is then supported in
$S_{<j}$. Moving terms, if necessary, from $P^{(j)}_{j,g}(x)$ to
$P^{(j)}_{j,<}(x)$ one may assume that $P^{(j)}_{j,g}(x)\in
R[x_{>j-1}]$.  Now update identity $(j)$ as follows: 
set
$P^{(j-1)}_{<j-1}(x)$ to be the terms in
$P^{(j)}_{<j}(x)+P^{(j)}_{j,<}(x)$ supported in $S_{<j-1}$; let 
$P^{(j-1)}_{j-1}(x)$  be the other terms in this sum;
set $P^{(j-1)}_{>j-1}(x)=
P^{(j)}_{>j}(x)+P^{(j)}_{j,g}(x)$. The stipulated conditions then
hold for the new display.

It follows that from identity $(m)$ above we can proceed to identity
$(0)$. However, then $P^{(0)}(x)=P^{(0)}_{>0}(x)\in gR[x]$.
\qedhere
\end{asparaenum}
\end{proof}


We now record the existence of an embedded resolution of singularities
that is  particularly  well adapted to computing with logarithmic
vector fields.

\begin{prp}\label{prp-log-res}
For every divisor $Y$ on $X$, there is an algorithm for construction
of an embedded resolution of singularities that is a composition of
blow-ups such that at each step the center of the blow-up is smooth
and a union of logarithmic strata of the total transform of $Y$ under
the previous blow-ups.
\end{prp}
\begin{proof}
By \cite{jarek}, there is an embedded resolution of singularities
$\pi\colon X'\to X$ that is a sequence of blow-ups
$\pi=\pi_k\circ\ldots\circ\pi_1$  with smooth
centers that furthermore is functorial: for any analytic isomorphism
$\iota\colon X\to X$ there is an analytic isomorphism $\tilde\iota
\colon X'\to X'$ such that
$\pi\circ\tilde\iota=\iota\circ\pi$.

Take $\delta\in\Der_X(-\log Y)$ and choose $\x_0\in X$ with
$\delta(\x_0)\neq 0$. By the Picard--Lindel\"of theorem, there is a
foliation of integrating curves $\gamma_\x(t)$ for $\delta$ at all
$\x$ near $\x_0$: $\gamma_\x(0)=\x$ and $\frac{\de}{\de t}(\gamma_\x(t))$
is $\delta$ evaluated at $\gamma_\x(t)$ for small $t$.  It follows
that the assignment $\Phi(\x,t)\colon \x\mapsto \gamma_\x(t)$ is
well-defined for $\x$ sufficiently near $\x_0$, and for all $t\ll
1$. Since $\delta$ is analytic and $\gamma_\x(t)$ is defined via an
integral, $\Phi$ is analytic. 

As $\delta$ is logarithmic along $Y$, $\x\in Y$ implies
$\gamma_\x(t)\in Y$. Hence, $\Phi(-,t)$ is a local automorphism of the
pair $(X,Y)$ near $\x\in Y$ for $t\ll 1$.  If $\pi_1\colon X'_1\to X$ is
the first blow-up in the resolution, then by functoriality $\Phi(-,t)$
lifts to an analytic family of local automorphisms of $(X'_1,
\pi^*Y)$. This is only possible if the center $C_1$ of the
blow-up is stable under $\delta$. This being true for all logarithmic
vector fields along $Y$, $C_1$ must be a union of logarithmic
strata of $Y$. The proposition follows by iterating the argument.
\end{proof}

\begin{cor}\label{cor-log-res}
If $Y\subseteq X$ is a divisor in a smooth $\CC$-scheme then there is
a resolution of singularities $\pi\colon X'\to X$ such that 
$\pi_*(\Der_{X'}(-\log\pi^*(Y)))=\Der_X(-\log Y)$ and
$\pi_*(\Omega^i_{X'}(\log \pi^*(Y)))=\Omega^i_X(\log Y)$.
\end{cor}
\begin{proof}
For
any divisor $Y_i$ and any smooth center $C_i\subseteq Y_i$ of codimension $2$ or
more on smooth $X_i$ the ideal quotient $\calI_{Y_i} \colon_{\calO_{X_i}} \calI_{C_i}$
equals $\calI_{Y_i}$. Thus Theorem \ref{thm-res-der} applies in each
step of the resolution obtained in Proposition \ref{prp-log-res}.
The final claim follows from Lemma \ref{lem-res-omega}.
\end{proof}

Not every logarithmic stratum will appear as a center in a resolution: 
\begin{exa}
Let $Y=\Var(xy(x+y)(x+ty))\subseteq \CC^3$. Then every point of the
$t$-axis is a logarithmic stratum. In any reasonable resolution of the
pair $(\CC^3,Y)$, of all points on the $t$-axis, the only
zero-dimensional canonical Whitney strata that feature as blow-up
centers are $(0,0,0)$ and $(0,0,1)$.
\schluss
\end{exa} 

\section*{Acknowledgments}
I would like to express my gratitude to Graham Denham for many
conversations about \cite{CDFV}, to Nero Budur and Luis
Narvaez-Macarro for very helpful comments and criticisms, and to
Viktor Levandovskyy for help with some computations using PLURAL
\cite{Plural}.

I am very grateful to Alex Dimca and Morihiko Saito for catching an
error in an earlier version, and to Morihiko Saito for generously
sharing his opinions on Theorem \ref{thm-jac-milnor} and Example
\ref{exa-ziegler}.

\def\scr{\mathcal}


\bibliographystyle{plain}
\bibliography{zieg}

\end{document}